\documentclass[11pt, letter, reqno]{amsart}
\usepackage[latin1]{inputenc}
\usepackage[english]{babel}
\usepackage{amsmath,amsthm}
\usepackage{amsfonts}
\usepackage{mathrsfs} 
\usepackage{latexsym}
\usepackage[dvips]{graphicx}
\usepackage{float}
\usepackage[natural]{xcolor}
\usepackage{algorithm}
\usepackage{algorithmicx}
\usepackage{algpseudocode}
\usepackage{slashed}
\usepackage{enumerate,enumitem}
\usepackage{multirow}
\usepackage[colorlinks,linkcolor=blue]{hyperref}
\usepackage{lineno}
\usepackage{setspace}
\usepackage{fullpage}
\usepackage[title]{appendix}
\usepackage{todonotes}
\usepackage{amssymb}
\usepackage{tikz} 
\usepackage{stmaryrd}
\usepackage{listings}
\usepackage{bm}
\usepackage{bbm}
\newcounter{propcounter}
\usepackage{enumitem}

\newcommand{\eps}{\varepsilon}

\theoremstyle{plain}
\newtheorem{thm}{Theorem}[section]
\newtheorem{theorem}[thm]{Theorem}
\newtheorem{conjecture}[thm]{Conjecture}
\newtheorem{lemma}[thm]{Lemma}
\newtheorem{corollary}[thm]{Corollary}
\newtheorem{proposition}[thm]{Proposition}

\theoremstyle{definition}

\newtheorem{question}[thm]{Question}
\newtheorem{problem}[thm]{Problem}
\newtheorem{remark}{Remark}

\newtheorem{definition}[thm]{Definition}
\newtheorem{claim}[thm]{Claim}
\newtheorem{fact}[thm]{Fact}

\newtheorem{example}[thm]{Example}

\newtheorem{defn-thm}[thm]{Definition-Theorem}
\numberwithin{equation}{section}

\newcommand{\btheorem}{\begin{theorem}}
\newcommand{\etheorem}{\end{theorem}}
\newcommand{\bconjecture}{\begin{conjecture}}
\newcommand{\econjecture}{\end{conjecture}}
\newcommand{\bproposition}{\begin{proposition}}
\newcommand{\eproposition}{\end{proposition}}
\newcommand{\bdefinition}{\begin{definition}}
\newcommand{\edefinition}{\end{definition}}
\newcommand{\bcorollary}{\begin{corollary}}
\newcommand{\ecorollary}{\end{corollary}}
\newcommand{\pr}{\begin{proof}}
\newcommand{\oof}{\end{proof}}
\newcommand{\bclaim}{\begin{claim}}
\newcommand{\eclaim}{\end{claim}}
\newcommand{\bquestion}{\begin{question}}
\newcommand{\equestion}{\end{question}}
\newcommand{\bfact}{\begin{fact}}
\newcommand{\efact}{\end{fact}}
\newcommand{\bremark}{\begin{remark}}
\newcommand{\eremark}{\end{remark}}
\newcommand{\eexample}{\end{example}}
\newcommand{\bexample}{\begin{example}}
\newcommand{\ma}{\end{lemma}}
\newcommand{\lem}{\begin{lemma}}

\begin{document}

\title{The minimum positive uniform Tur\'an density in uniformly dense $k$-uniform hypergraphs}

\author{Hao Lin}
\address{School of Mathematics, Shandong University,
Jinan, China.  Email: {\tt lhao17@mail.sdu.edu.cn}.}

\author{Guanghui Wang}
\address{School of Mathematics, Shandong University,
Jinan, China. Email: {\tt ghwang@sdu.edu.cn}.}
\thanks{The second author is supported by the Natural Science Foundation of China (12231018) and the Young Taishan Scholars Program of Shandong Province (201909001).}

\author{Wenling Zhou}
\address{School of Mathematics, Shandong University, Jinan, China, and Laboratoire Interdisciplinaire des Sciences du Num\'{e}rique, Universit\'{e} Paris-Saclay, Orsay, France.  Email: {\tt gracezhou@mail.sdu.edu.cn}.}

\keywords {Uniform Tur\'an density, Quasi-random hypergraph,  Hypergraph regularity method}

\begin{abstract}
A $k$-graph (or $k$-uniform hypergraph) $H$ is \emph{uniformly dense} if the edge distribution of $H$ is uniformly dense with respect to every large collection of $k$-vertex cliques induced by sets of $(k-2)$-tuples.  Reiher, R\"odl and Schacht [\textit{Int. Math. Res. Not., 2018}] proposed the study of the uniform  Tur\'an density $\pi_{k-2}(F)$ for given $k$-graphs $F$ in uniformly dense $k$-graphs.
Meanwhile, they [\textit{J. London Math. Soc., 2018}] characterized $k$-graphs $F$ satisfying $\pi_{k-2}(F)=0$  and showed that $\pi_{k-2}(\cdot)$ ``jumps" from 0 to at least $k^{-k}$.
In particular, they asked whether
there exist $3$-graphs $F$ with $\pi_{1}(F)$ equal or arbitrarily close to $1/27$.
Recently, Garbe, Kr\'al' and Lamaison 
[\textit{arXiv:2105.09883}]
constructed some $3$-graphs with $\pi_{1}(F)=1/27$.

In this paper, for any $k$-graph $F$, we give a lower bound of $\pi_{k-2}(F)$ based on a probabilistic framework, and provide a general theorem that reduces proving an upper bound on $\pi_{k-2}(F)$ to embedding $F$ in reduced $k$-graphs  of the same density using the regularity method for $k$-graphs. By using this result and Ramsey theorem for multicolored hypergraphs, we extend the results of Garbe, Kr\'al' and Lamaison to $k\ge 3$. 
In other words, we give a  sufficient condition for  $k$-graphs $F$ satisfying $\pi_{k-2}(F)=k^{-k}$. Additionally, we also construct an infinite family of $k$-graphs with $\pi_{k-2}(F)=k^{-k}$. 
\end{abstract}

\maketitle

\section{Introduction}
For a positive integer $\ell$, we denote by $[\ell]$ the set $\{1,\dots,\ell\}$.
Given $k\geq 2$, for a finite set $V$, we use  $[V]^k$ to denote the collection of all subsets of $V$ of size $k$, and $V^{[k]}$ to denote the Cartesian power $V\times \dots \times V$.
We may drop one pair of brackets and write $[\ell]^k$ instead of $[[\ell]]^k$.
A {\it $k$-uniform hypergraph} $H$ (or \textit{$k$-graph} for short)  is a pair $H=(V(H),E(H))$ where $V(H)$ is a finite set of \textit{vertices} and $E(H)\subseteq [V(H)]^{k}$ is a set of \textit{($k$-)edges}. 
A {\it $k$-uniform clique} of order $\ell\ge k$, denoted by $K^{(k)}_{\ell}$, is a $k$-graph on $\ell$ vertices consisting of all $\binom{\ell}{k}$
different $k$-tuples.
So a 2-graph is a simple graph, and a $2$-uniform clique is a complete graph.

\subsection{Tur\'an problems in hypergraphs}
The Tur\'an problem introduced by Tur\'an~\cite{Tu-graph} asks to study for a given $k$-graph $F$ its \textit{Tur\'an number} $ex(n,F)$, the maximum number of $k$-edges in an $F$-free $k$-graph on $n$ vertices. It is a long-standing open problem in Extremal Combinatorics to develop some understanding of these numbers for general $k$-graphs. Ideally, one would like to compute them exactly, but even asymptotic results are currently only known in certain cases, see a wonderful survey~\cite{Keevash-survey}.
It is well known and not hard to observe that the sequence $ex(n,F)/\binom{n}{k}$ is decreasing.
Thus, one often focuses on the \textit{Tur\'an density} $\pi(F)$ of $F$ defined by
\[
\pi(F)=\lim_{n\to \infty}\frac{ex(n,F)}{\binom{n}{k}}.
\]
Tur\'an densities are well-understood for graphs.
Indeed, the Mantel's theorem~\cite{mantel1907problem} and the  Tur\'an's theorem~\cite{Tu-graph} gave the Tur\'an number of complete graphs exactly,
and Erd\H{o}s and Stone~\cite{E-Stone} (also see Erd\H{o}s and Simonovits~\cite{E-Simonovits}) determined the Tur\'an density of any graph $F$ to be equal to 
$\frac{\chi(F)-2}{\chi(F)-1}$,
where $\chi(F)$ denotes the {\it chromatic number} of $F$,
that is the minimum number of colors used to color $V(F)$ such that any two adjacent vertices receive distinct colors. However, the analogous questions for hypergraphs are notoriously difficult, even for the 3-graphs case.
Despite much efforts and attempts so far, our knowledge is somewhat
limited, such as the Tur\'an density of 3-uniform clique $K^{(3)}_4$ on four vertices,  raised by Tur\'an in 1941, is still open~\cite{k43-1, k43-2}. The only general theorem in this area due to Erd\H{o}s~\cite{r-partite} asserts the following result.

\begin{theorem} [{\cite[Theorem 1]{r-partite}}] 
For $k \ge 2$, a $k$-graph $F$ satisfies $\pi(F)=0$ if and only if it is $k$-partite, i.e., there is a partition $V_1\cup V_2 \cup \dots \cup  V_k$ of $V(F)$ such that every edge of $F$ contains
precisely one vertex from each $V_i$ for $i\in [k]$.
\end{theorem}

An important reason for the extreme difficulty in the Tur\'an problems of  hypergraphs is the existence
of certain quasi-random  hypergraphs (some hypergraphs with positive
density obtained from random tournaments or random colorings
of complete graphs) avoiding given subhyperhgraphs.
More precisely, a $k$-graph $H=(V,E)$ is {\it quasi-random} with density $d>0$ if every subset  $U\subseteq V$ satisfies 
\begin{equation*}
\left|\Big|\binom{U}{k}\cap E\Big|-d \binom{|U|}{k}\right|=o(|V|^k).
\end{equation*}
The main result in~\cite{chung1990quasi} asserts, quasi-random graphs with positive density contain a correct number of copies of arbitrary graphs $F$ of fixed size, namely, the number
of copies of $F$ is as expected in the random graph with the same density.
As mentioned above, the Tur\'an problems for quasi-random $k$-graphs with $k\ge 3$, is quite different from the case $k = 2$ and has been
an important topic over decades.
Note that for Tur\'an-type problems,
it is sufficient to require only 
for a $k$-graph $H=(V,E)$ a lower bound of the form
\begin{equation}\label{vertex-dense}
\left|\binom{U}{k}\cap E\right|\ge d\binom{|U|}{k}-\mu|V|^k,
\end{equation}
to hold for any $U\subseteq V$ and $\mu>0$. In general, a $k$-graph $H$ satisfying the condition~\eqref{vertex-dense} is said to be {\it $(d, \mu, 1)$-dense}  (or {\it uniformly dense}).
A somewhat standard application of the so-called \emph{weak regularity lemma} for hypergraphs (straightforward extension of Szemer\'edi's regularity lemma for graphs~\cite{szemeredi1975regular}) implies that such a $(d, \mu, 1)$-dense $k$-graph always contains a quasi-random subhypergraph of density $d$.
Therefore, this suggests a systematic
study of Tur\'an problems in uniformly dense hypergraphs.

\subsection{Tur\'an problems in uniformly dense hypergraphs}
In 1982, Erd\H{o}s and S\'os~\cite{E-Sos} was the first to raise questions on the Tur\'an densities in uniformly dense $3$-graphs. 
Specifically, the Tur\'an problems about the optimal density  in  uniformly dense 3-graphs not
containing a given $3$-graph $F$ (such as $K^{(3)}_4$) can be made precise by introducing the quantities
\begin{align*}
\pi_{1}(F) = \sup \{ d\in [0,1] & : \text{for\ every\ } \mu>0 \ \text{and\ } n_0\in \mathbb{N},\ \text{there\ exists\ an\ } F \text{-free} \\
&\quad (d,\mu,1)\text{-dense}~ 3 \text{-graph~} H\ \text{with~} |V(H)|\geq n_0 \}.
\end{align*}
With this notation at hand, Erd\H{o}s and S\'os  asked to determine $\pi_1(K^{(3)}_4)$ and $\pi_1(K^{(3)-}_4)$, where $K^{(3)-}_4$ is $K^{(3)}_4$ with an edge removed. However, determining $\pi_1(F)$  for a given 3-graph $F$ is also very challenging.
The  conjecture for $\pi_1(K^{(3)}_4)=1/2$  has been an urgent problem in this area since R\"odl~\cite{k43-rodl} gave a quasi-random construction in 1986.
$\pi_1(K^{(3)-}_4)=1/4$
was solved recently by Glebov, Kr\'al' and Volec~\cite{k43minus-1}, and independently by Reiher, R\"odl and
Schacht~\cite{k43minus-2}. We refer the reader to the survey by Reiher~\cite{reiher2020extremal} for a more comprehensive treatment and further results for $3$-graphs.

The study of Tur\'an problems in uniformly dense $k$-graphs has recently gained popularity due to the work of  Reiher,
R\"odl and Schacht~\cite{vanishing, RRS-Mantel,  k43minus-2}.
In addition to providing a solution to the aforementioned conjecture of Erd\H{o}s and S\'os,
they also determined a large collection of uniform Tur\'an densities of $k$-graphs  based on a family of naturally defined uniformly dense conditions.
 Here we state a concept of $(d,\mu, j)$-dense $k$-graphs (see Definition~\ref{def-j-dense}) considered by Reiher, R\"odl and
Schacht in~\cite{RRS-Mantel}, 
which serves as a natural generalization of $(d, \mu, 1)$-dense $3$-graphs.

Given integers $k>j\ge 0$ and a $j$-graph $G^{(j)}$, we denote by $\mathcal{K}_k(G^{(j)})$ for the collection of $k$-sets of $V(G^{(j)})$ which span a $j$-uniform clique $K^{(j)}_k$ on $k$ vertices in $G^{(j)}$.
Note that $|\mathcal{K}_k(G^{(j)})|$ is the number of all copies of $K^{(j)}_k$ in $G^{(j)}$.

\begin{definition}[$(d, \mu, j)$-denseness] \label{def-j-dense}
	Given integers $n\ge k>j\ge 0$, let real numbers $d\in [0,1]$, $\mu>0$, and $H = (V, E)$ be a
	$k$-graph with $n$ vertices. We say that $H$ is {\it $(d, \mu, j)$-dense} if
\begin{equation}\label{eq-j-dense}
\left|\mathcal{K}_k(G^{(j)})\cap E\right|\ge d\left|\mathcal{K}_k(G^{(j)})\right|-\mu n^k
\end{equation}
holds for all $j$-graphs $G^{(j)}$ with vertex set $V$.
\end{definition}

\begin{remark}\label{rem-0-dense}
	Note that for any vertex set $V$ there are only two $0$-graphs (the one with empty edge
	set and the one with the empty set being an edge).
	Therefore, in the degenerate case,  $H$ is $(d, \mu, 0)$-dense if
\[
|E|\ge d\binom{|V|}{k}-\mu n^k.
\]
\end{remark}

Restricting to $(d,\mu,j)$-dense $k$-graphs, the appropriate
{\it uniform Tur\'an density} $\pi_{j}(F)$ for a given $k$-graph $F$ can be defined as
\begin{equation*}
	\begin{split}
		\label{j-turan-dense}
		\pi_{j}(F) = \sup \{ d\in [0,1] & : \text{for\ every\ } \mu>0 \ \text{and\ } n_0\in \mathbb{N},\ \text{there\ exists\ an\ } F \text{-free} \\
		&\quad (d,\mu,j)\text{-dense}~ k \text{-graph~} H\ \text{with~} |V(H)|\geq n_0 \}.
	\end{split}
\end{equation*}

In particular, Reiher, R\"odl and Schacht~\cite{RRS-Mantel} proposed the following problem.

\begin{problem}[{\cite[Problem 1.7]{RRS-Mantel}}]\label{problem1}
	Determine $\pi_j(F)$ for all $k$-graphs $F$ and all $0\le j\le k-2$.
\end{problem}

\begin{remark}
For $j=k-1$, it is known that every $k$-graph $F$ satisfies $\pi_{k-1}(F)=0$,  which follows from the work in~\cite{count-dense}. Moreover, for every $k$-graph $F$ we have
\begin{equation}\label{compare}
	\pi(F)=\pi_0(F)\ge \pi_{1}(F)\ge \dots \ge \pi_{k-2}(F)\ge \pi_{k-1}(F)=0,
\end{equation}
since $\mathcal K_k(G^{(j)}) = \mathcal K_k(G^{(j+1)})$ for every $j$-graph $G^{(j)}$ with $G^{(j+1)} = \mathcal K_{j+1}(G^{(j)})$, and $\pi(F)=\pi_0(F)$ by Remark~\ref{rem-0-dense}. 
\end{remark}

Given a $k$-graph $F$, the quantities appearing in this chain of inequalities~(\ref{compare}) will probably be harder to determine the further they are on the left. This suggests that Problem~\ref{problem1} for the case $j=k-2$ is the first interesting case and we will focus on $\pi_{k-2}(F)$  in this paper.

In 2018, Reiher, R\"odl and Schacht~\cite{vanishing} suggested that for the case $j=k-2$ one can establish a theory that resembles some extent the classical theory for graphs initiated by Tur\'an himself and developed further by Erd\H{o}s, Stone, Simonovits and many others. 
In particular, they gave a characterization of $k$-graphs $F$
with $\pi_{k-2}(F)=0$ (see Theorem~\ref{thm-k-2-zero-density}). 
In addition to $k$-graphs $F$ with  $\pi_{k-2}(F)=0$, the only $k$-graph for which $\pi_{k-2}(\cdot)$ is known is $F^{(k)}$ on $(k + 1)$ vertices with three edges, and $\pi_{k-2}(F^{(k)})=2^{1-k}$ is obtained from~\cite{RRS-Mantel}.

For simplicity,  we write $\llbracket {i_1,i_2, \dots, i_\ell}\rrbracket$ to denote a set  $\{i_1,i_2, \dots, i_\ell\}\subset \mathbb Z$ with $i_1<i_2<\dots<i_\ell$, and $\llbracket v_{i(1)}, v_{i(2)},\dots, v_{i(\ell)}\rrbracket$ to denote a set  $\{ v_{i(1)}, v_{i(2)},\dots, v_{i(\ell)}\}$ with ${i(1)}< {i(2)} \dots< {i(\ell)}$.
Given a $k$-graph $F$ with $f$ vertices,
let $\partial F:=\{S\in [V(F)]^{k-1}: \exists~ e\in E(F), S\subset e\}$ be the \emph{shadow} of $F$. 
We say that an  ordering $( v_1,v_2,\dots,v_f)$ of $V(F)$ is {\it vanishing} if $\partial F$ can be partitioned into $k$ disjoint sets $\mathcal C_\ell $ for $\ell\in [k]$ such that every $k$-edge $e=\llbracket v_{i(1)},\dots, v_{i(k)} \rrbracket$ of $F$ satisfies 
\[
e\setminus \{v_{i(\ell)}\}\in \mathcal C_\ell, ~\text{for~every~} \ell\in [k].
\]
These $(k-1)$-sets that belong to $\mathcal C_\ell $  are referred to as {\it $\ell$-type} w.r.t. (i.e., with respect to) $F$ (under the vanishing  ordering). In particular,
given a vanishing ordering $\tau$ of $V(F)$ and a $(k-1)$-set $S\subset V(F)$, we say that $S$ is {\it $\ell$-type} w.r.t. $F$, if there is a vertex $v\in V(F)$ such that $S\cup \{v\}$ is a $k$-edge of $F$ and the ordering of $S\cup \{v\}$ under the $\tau$ is such that $v$ is in the $\ell$\textsuperscript{th} position.

\begin{theorem}[{\cite[Theorem 6.1]{vanishing}}]\label{thm-k-2-zero-density}
A $k$-graph $F$ satisfies $\pi_{k-2}(F)=0$ if and only if it has a vanishing ordering of $V(F)$.
\end{theorem}

In fact, Theorem~\ref{thm-k-2-zero-density} yields the following strengthening. For $n\in \mathbb N$, consider a uniform random partition of $[n]^{k-1}$ into the sets $\mathcal C'_\ell $ for $\ell\in [k]$. We define a probability distribution $H(n)$ on $k$-graphs of order $n$ as follows.  Let $V(H(n))=[n]$ and include a  $k$-set $e=\llbracket i_1, \dots, i_k\rrbracket$ in $E(H(n))$ if $e$ satisfies that 
$e\setminus \{i_{\ell}\}\in \mathcal C'_\ell$ for every $\ell\in [k]$. Using probabilistic arguments, we can show that for any fixed $\mu>0$ and large $n$ there exists $H \in  H(n)$ such that $H$ is $(k^{-k}, \mu, k-2)$-dense. Clearly, each subhypergraph of $H$ has a vanishing ordering of its vertices. 
Thus, Theorem~\ref{thm-k-2-zero-density} implies the following result.

\begin{corollary} \label{lower-bound}
If a $k$-graph $F$ satisfies $\pi_{k-2}(F)>0$, then $\pi_{k-2}(F)\ge k^{-k}$.
\end{corollary}

Therefore, Reiher, R\"odl and Schacht~\cite{vanishing} proposed the following problems for $k=3$.

\begin{problem}\label{problem2}
Is there a $k$-graph $F$ with $\pi_{k-2}(F)$ equal or arbitrarily close to $k^{-k}$?
\end{problem}

For $k=2$, the answer to Problem~\ref{problem2} is no, since $\pi_{0}(F)=\pi(F)$ and every graph $F$ with $\pi(F)> 0$ satisfies $\pi(F)\ge 1/2$ by the result in~\cite{E-Stone}.
However, recently Garbe, Kr\'al' and Lamaison~\cite{1/27} gave an affirmative answer to Problem~\ref{problem2} for $k=3$ by giving a sufficient condition for 3-graphs $F$ with $\pi_{1}(F)=1/27$, and constructing examples of 3-graphs that satisfy this condition.
%

\subsection{Our results}
In this paper, for any $k\ge 3$, we first study the upper and lower bounds of $\pi_{k-2}(F)$ for any given graph $F$ within a global framework.
Upon reviewing all the known results for $\pi_{k-2}(\cdot)$, we observe that the lower bounds of $\pi_{k-2}(\cdot)$ are all obtained from probabilistic constructions. In particular, when $k=3$, the lower bounds of $\pi_{1}(\cdot)$ are based on the probabilistic framework introduced in~\cite[Section 2]{reiher2020extremal}, which is inspired by and unifies earlier probabilistic constructions, in particular the one from~\cite{k43-rodl}. We summarize this framework in the following theorem.

\begin{theorem}\label{low-3-graphs}
Let $F$ be a $3$-graph. Suppose that there exists $r\in \mathbb N$ and a set $\mathscr P \subseteq [r]\times[r]\times[r]$
with the following properties: for every $n \in \mathbb N$ and every $\psi :
[n]^2 \to [r]$, the $3$-graph $H$ with vertex set
$[n]$ and edge set
\[
E(H)=\{\{x,y,x\}\in [n]^3: x<y<z \text{~and~} (\psi(y,z), \psi(x,z), \psi(x,y))\in \mathscr P\}
\] 
 is $F$-free. Then,
$\pi_1(F)\ge |\mathscr P|/r^3$.
\end{theorem}

Using Azuma-Hoeffding inequality,
we extend the above framework  and obtain a lower bound of $\pi_{k-2}(\cdot)$ based on a
more general framework for all $k\ge 3$.

\begin{theorem}\label{low-k-graphs}
Let $F$ be a $k$-graph. Suppose that there exists $r\in \mathbb N$ and a set $\mathscr P \subseteq [r]^{[k]}$
	with the following properties: for every $n \in \mathbb N$ and every $\psi :
	[n]^{k-1} \to [r]$, the $k$-graph $H$ with vertex set
	$[n]$ and edge set
	\[
	E(H)=\{e=\llbracket {i_1,i_2, \dots, i_\ell}\rrbracket \in [n]^k: (\psi(e\setminus \{i_{1}\}), \psi(e\setminus \{i_{2}\}), \dots, \psi(e\setminus \{i_{k}\}))\in \mathscr P\}
	\] 
	is $F$-free. Then,
	$\pi_{k-2}(F)\ge |\mathscr P|/r^k$.
\end{theorem}

Clearly, Theorem~\ref{low-k-graphs} is equivalent to Theorem~\ref{low-3-graphs} when $k = 3$. Next we will provide a general statement that reduces proving an upper bound of $\pi_{k-2}(F)$ for a given  $k$-graph $F$ to embedding $F$ in \emph{reduced $k$-graphs} (see Definition~\ref{def-reduced-graph})  of the same density using the regularity method for $k$-graphs. We start with some notation introduced in ~\cite[Section 4]{RRS-Mantel}.

\begin{definition}[reduced $k$-graphs]\label{def-reduced-graph}
Given a finite index set $I\in \mathbb N$ with $|I|=m$, for each $\mathcal X\in [I]^{k-1}$, let $\mathcal {P}_{\mathcal X}$ denote a finite nonempty vertex set such that for any two distinct $\mathcal X, \mathcal X'\in [I]^{k-1}$ the sets $\mathcal {P}_{\mathcal X}$ and $\mathcal {P}_{\mathcal X'}$ are disjoint.
For any $\mathcal Y \in [I]^{k}$, let $\mathcal A_{\mathcal Y}$ denote a $k$-partite 
$k$-graph with vertex partition $\{\mathcal {P}_{\mathcal X}: \mathcal X \in [{\mathcal Y}]^{k-1}\}$.
Then the $\binom{|I|}{k-1}$-partite 
	$k$-graph $\mathcal A$ with
	\[
	V(\mathcal A)=\bigcup_{\mathcal X \in [I]^{k-1}}\mathcal {P}_{\mathcal X}, ~\text{and}~~ E(\mathcal A)=\bigcup_{\mathcal Y\in [I]^{k}} E(\mathcal A_{\mathcal Y})
	\]
is called an {\it $m$-reduced $k$-graph} with \emph{index set} $I$, \emph{vertex classes} $\mathcal {P}_{\mathcal X}$ and {\it constituents}  $\mathcal A_{\mathcal Y}$. 
\end{definition}

For brevity, we often simply write ``let $\mathcal A$ be an $m$-reduced $k$-graph"  instead of ``let $\mathcal A$ be an $m$-reduced $k$-graph with index set $[m]$, vertex classes $\mathcal {P}_{\mathcal X}$ and
constituents $\mathcal A_{\mathcal Y}$". 
Given a $m$-reduced $k$-graph $\mathcal A$ and  $d \in[0, 1] $, we say that 
$\mathcal A$ is {\it $d$-dense} if 
\[
|E(\mathcal A_{\mathcal Y})|\ge d \cdot \prod_{{\mathcal X}\in [\mathcal Y]^{k-1}}|\mathcal {P}_{\mathcal X}|
\]
holds for all $\mathcal Y \in [m]^{k}$.

Whether an $m$-reduced $k$-graph $\mathcal A$ can ``embed" a given $k$-graph $F$ can be expressed in terms of the existence of so-called ``reduced maps" which are going to be introduced next.

\begin{definition}[reduced maps]\label{def-reduce-map}
A \emph{reduced map} from a $k$-graph $F$ to a reduced $k$-graph $\mathcal A=(I, \mathcal {P}_{\mathcal X}, \mathcal A_{\mathcal Y})$ is a pair $(\phi,\psi)$ such that
\begin{enumerate}[label=(\arabic*)]
\item $\phi:V(F)\to I$ and $\psi: \partial F\to V(\mathcal A)$;
\item if $S=\{i_1,i_2,\dots, i_{k-1}\}\in \partial F$, then $\mathcal X=\{\phi(i_1), \phi(i_2), \dots, \phi(i_{k-1})\}\in [I]^{k-1}$ and $\psi(S) \in \mathcal P_{\mathcal X}$;
\item if $e=\{i_1,i_2,\dots, i_{k} \} \in E(F)$, then $\mathcal Y= \{\phi(i_1), \phi(i_2), \dots, \phi(i_{k})\}\in [I]^{k}$ and
\[
\{\psi(e\setminus \{i_{1}\}), \psi(e\setminus \{i_{2}\}), \dots, \psi(e\setminus \{i_{k}\})\}\in E( \mathcal A_{\mathcal Y}).
\]
\end{enumerate}
\end{definition} 

If there is a reduced map from $F$ to $\mathcal A$, we say that $\mathcal A$  \emph{embeds} $F$.
Now the general result about proving an upper bound of $\pi_{k-2}(F)$ for a given  $k$-graph $F$ in reduced $k$-graphs asserts the following.

\begin{theorem} \label{thm-turan-reduced}
	Let $F$ be a $k$-graph with $k\ge 3$ and $d\in [0,1]$.
	If for any $\eps>0$ there exists $m\in \mathbb N$ such that each $(d+\eps)$-dense $m$-reduced $k$-graph $\mathcal A$ embeds $F$, then $\pi_{k-2}(F)\le d$.
\end{theorem}

We also remark that parts of the proof of this result are implicit in~\cite{RRS-Mantel}. Still, we believe it to be useful
to gather the argument in its entirety. Theorem~\ref{low-k-graphs} and Theorem~\ref{thm-turan-reduced} serves as a general tool for the Tur\'an problem in $(d,\mu,k-2)$-dense $k$-graphs. In particular, when $k=3$, this tool is widely used in~\cite{4/27,1/27,vanishing,k43minus-2}.

Next, inspired by the research of Garbe, Kr\'al' and Lamaison~\cite{1/27}, we answer Problem~\ref{problem2}
by giving a non-trivial sufficient condition for  $k$-graphs $F$ satisfying $\pi_{k-2}(F)=k^{-k}$, and construct an infinite family of $k$-graphs with $\pi_{k-2}(F)=k^{-k}$.

\begin{theorem}\label{1/kthm}
Given $k\ge 3$, let $F$ be a $k$-graph satisfying the following conditions:
\stepcounter{propcounter}
\begin{enumerate}[label=$(\roman*)$]
\item[{\rm ($\clubsuit$)}] \label{th-condition-1}
$F$ has no vanishing ordering of $V(F)$; 
\item[{\rm ($\spadesuit$)}] \label{th-condition-2} For each pair $\{i,j\}\in [k]^2$ with $i<j$, $F$ can always be partitioned into two spanning subhypergraphs $F^1_{i,j}$ and $F^2_{i,j}$ such that there exists an ordering of $V(F)$ that is vanishing both for $F^1_{i,j}$ and $F^2_{i,j}$
and for any two edges $e_1\in E(F^1_{i,j})$ and $e_2\in E(F^2_{i,j})$ with $|e_1\cap e_2|=k-1$,  $e_1\cap e_2$ is $i$-type w.r.t. $F^1_{i,j}$ and $j$-type w.r.t. $F^2_{i,j}$. 
	\end{enumerate}
	Then $\pi_{k-2}(F)=k^{-k}$.
\end{theorem}

%
%

Furthermore, based on Theorem~\ref{1/kthm}, we construct an infinite family of $k$-graphs $F$ which satisfy the conditions given in Theorem~\ref{1/kthm}.

\begin{theorem}\label{example-thm} 
Given integers $t\ge k-2\ge 1$,  let $F^{(k)}_t$ be the $k$-graph consisting of $(3t+k+2)$ vertices $a_1, a_2, \dots, a_{k-1}$, $b_0,  b_1, \dots, b_t, c_0,  c_1, \dots, c_t, d_0,  d_1, \dots, d_t$ and the following $3(t+2)$ edges:
	\begin{center} 
		$a_1\dots a_{k-1}b_0,~~ a_2\dots a_{k-1}b_0b_1, ~\dots, ~ a_{k-1}b_0\dots b_{k-2}, ~b_0b_1\dots b_{k-1},~\dots, ~b_{t-k+1}\dots b_t, ~b_{t-k+2}\dots b_tc_t,$
		
		$a_1\dots a_{k-1}c_0,~~ a_2\dots a_{k-1}c_0c_1, ~\dots, ~ a_{k-1}c_0\dots c_{k-2}, ~c_0c_1\dots c_{k-1},~\dots, ~c_{t-k+1}\dots c_t, ~c_{t-k+2}\dots c_td_t,$
		
		$a_1\dots a_{k-1}d_0,~~ a_2\dots a_{k-1}d_0d_1,~\dots,~ a_{k-1}d_0\dots d_{k-2}, ~d_0d_1\dots d_{k-1},~\dots, ~d_{t-k+1}\dots d_t, ~d_{t-k+2}\dots d_tb_t.$
	\end{center}
	We have $\pi_{k-2}(F^{(k)}_t)=k^{-k}$.
\end{theorem}

\begin{remark}
A {\it tight $k$-uniform path} of length $\ell \ge k$, is a sequence $(v_1, v_2 \dots, v_\ell)$ of distinct vertices, satisfying  that $\{v_i, \dots, v_{i+k-1}\}$ is an edge for every $i\in [\ell-k+1]$. Clearly, the $k$-graph $F^{(k)}_t$ in the statement of Theorem~\ref{example-thm} can be viewed as consisting of the following three tight $k$-uniform  paths of length $(t+k+1)$:
\[
(a_1, \dots,  a_{k-1}, b_0, b_1, \dots, b_t, c_t),
(a_1, \dots,  a_{k-1}, c_0, c_1, \dots, c_t, d_t) \text{~and~}
(a_1, \dots,  a_{k-1}, d_0, d_1, \dots, d_t, b_t).
\] 
In particular, when $k=3$, the $3$-graphs $F^{(3)}_t$ are exactly the family of  $3$-graphs given by Garbe, Kr\'al' and Lamaison~\cite{1/27}.
In addition, the smallest $k$-graph $F^{(k)}_t$  has $(4k-4)$ vertices and $3k$ edges (see Figure~\ref{figure-1}).
\end{remark}

\begin{figure}
\begin{center}
\begin{tikzpicture}
[inner sep=2pt,
vertex/.style={circle, draw=blue!50, fill=blue!50},
]
\filldraw[black] (0,0) circle (3pt);
\filldraw[black] (2,0) circle (3pt);
\filldraw[black] (3.8,0) circle (3pt);

\filldraw[black] (6,1) circle (3pt);
\filldraw[black] (8,1) circle (3pt);
\filldraw[black] (10,1) circle (3pt);

\filldraw[black] (6,0) circle (3pt);
\filldraw[black] (8,0) circle (3pt);
\filldraw[black] (10,0) circle (3pt);

\filldraw[black] (6,-1) circle (3pt);
\filldraw[black] (8,-1) circle (3pt);
\filldraw[black] (10,-1) circle (3pt);

\draw[line width=5pt,red,opacity=0.5] (0, 0) .. controls (2,0) and (5.5,0).. (6,1);
\draw[line width=5pt,red,opacity=0.5] (6,1) --  (10,1);
\draw[line width=5pt,red,opacity=0.5] (10,1) --  (10,0);

\draw[line width=5pt,blue,opacity=0.5] (0, 0) --  (10,0);
\draw[line width=5pt,blue,opacity=0.5] (10,0) --  (10,-1);

\draw[line width=5pt,green,opacity=0.5] (0, 0) .. controls (2,0) and (5.5,0).. (6,-1);
\draw[line width=5pt,green,opacity=0.5] (6,-1) --  (10,-1);
\draw[line width=5pt,green,opacity=0.5] (10,-1) .. controls (11,0) .. (10,1);

\node at (0,-0.5) {$a_1$};
	\node at (2,-0.5) {$a_2$};
	\node at (3.8,-0.5) {$a_3$};
	
		\node at (6,0.5) {$b_0$};
		\node at (8,0.5) {$b_1$};
	\node at (10.2,1.3) {$b_2$};
			
		\node at (6,-0.5) {$c_0$};
		\node at (8,-0.5) {$c_1$};
	\node at (9.8,-0.5) {$c_2$};
	
	\node at (6,-1.5) {$d_0$};
	\node at (8,-1.5) {$d_1$};
\node at (10.2,-1.3) {$d_2$};

\end{tikzpicture}
\caption{An illustration of the smallest $k$-graph $F^{(k)}_t$  for the case $k=4$ and $t=2$.}\label{figure-1}
\end{center}
\end{figure}
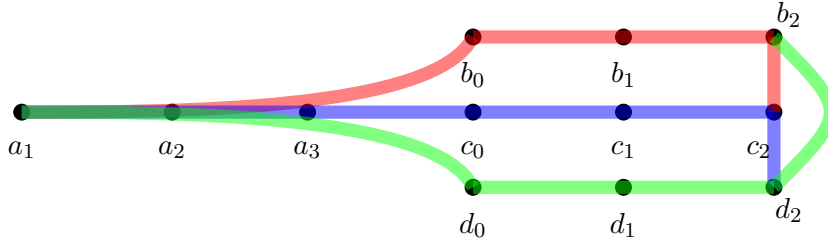

\subsection*{Organization}
The rest of this paper is organized as follows. 
In the next section, we give a probabilistic construction to prove Theorem~\ref{low-k-graphs}.
A key tool in the proof of Theorem~\ref{thm-turan-reduced} is the hypergraph regularity method. 
Therefore, in the Section~\ref{sec-regular-method}, we will review the regularity method for $k$-graphs, which as an extension of Szemer\'edi's regularity lemma for graphs, has been a celebrated tool for embedding problems in hypergraphs. We use a popular version of the regularity lemma for $k$-graphs due to R\"{o}dl and Schacht~\cite{regularity-lemmas} (a similar result was proved earlier by Gowers~\cite{G07}),
and with it we derive a ``clean" version of the regularity lemma for $k$-graphs (see Corollary~\ref{clean}).
In Section~\ref{sec-Reduced-pf}, we will 
give the proof of Theorem~\ref{thm-turan-reduced} using Corollary~\ref{clean} and an embedding lemma from~\cite{Embeddings}. 
In Section~\ref{sec-Reduced-embed}, we  prove a number of auxiliary
results and use them to prove an embedding lemma of reduced $k$-graphs (see Lemma~\ref{lem-Embedding-reduced}), which is the key to prove
Theorem~\ref{1/kthm}.
Finally, in Section~\ref{sec-example}, we give an equivalent transformation of vanishing ordering, and combine with Theorem~\ref{1/kthm} to give the proof of Theorem~\ref{example-thm}.
Some remarks and open problems will be given in the last section.

\section{Proof of Theorem~\ref{low-k-graphs}}
In this section, we shall prove Theorem~\ref{low-k-graphs}. To do this, we need the following lemma, also known as the Azuma-Hoeffding inequality from~\cite[Corollary 2.27]{random-graph}.

\begin{lemma}\label{Azuma}
	Let $Z_1, \dots, Z_n$ be independent random variables, with $Z_i$ taking values in a set $C_i$ for $i\in [n]$. Assume that a function $f: C_1\times C_2 \times \dots \times C_n \to \mathbb R$ satisfies the following Lipschitz condition for some number $c_i$:
	\begin{enumerate}
		\item[{\rm (L)}] If two vectors $\mathbf{z}, \mathbf{z'}\in \prod^n_1C_i$ differ only in the $i$th coordinate, then $|f(\mathbf{z})-f(\mathbf{z'})|\le c_i$.
	\end{enumerate}
	Then, the random variable $Y=f(Z_1, \dots, Z_n)$ satisfies, for any $\eta\ge 0$
	\[
	\mathbb P(Y\le \mathbb E (Y)-\eta)\le \exp (-\frac{\eta^2}{2\sum_1^nc_i^2}).
	\]
\end{lemma}

Now we prove Theorem~\ref{low-k-graphs} using the following construction.
\begin{proof}[Proof of Theorem~\ref{low-k-graphs}]
Let $F$ be a $k$-graph satisfying the statement given in Theorem~\ref{low-k-graphs}.
For any $n\in \mathbb N$,
we consider $\psi: [n]^{k-1} \to [r]$ as a
random $r$-coloring with each color associated to a $(k-1)$-set with probability $1/r$ independently and uniformly. We now define a probability distribution $H(n)$ on $k$-graphs of order $n$ as follows.  Let $V(H(n))=[n]$, and include a $k$-set $e=\llbracket x_1,\dots, x_k\rrbracket \in  [n]^k$ in $E(H(n))$ if $e$ satisfies 
\[ \left( \psi(e\setminus \{x_{1}\}), \psi(e\setminus \{x_{2}\}), \dots, \psi(e\setminus \{x_{k}\}) \right)\in \mathscr P.
\] 
Let $E=E(H(n))$ be the random set of $k$-edges of $H(n)$.
Observe that for each $k$-set $e\in [n]^k $, the probability of the event ``$e\in E$" is $|\mathscr P|/r^k$.
Moreover, for every $(k-1)$-set $X_t\in [n]^{k-1}$ with $1\le t\le  \binom{n}{k-1}$, we can view $\psi (X_{t})$ as an independent random variable with $\phi (X_{t})$ taking values in set $[r]$.
For each $(k-2)$-graph $G^{(k-2)}$ on vertex set $[n]$, let $Y$ denote the random variable $|\mathcal{K}_k(G^{(k-2)})\cap E(H(n))|$. Then $Y$ may be regarded as a function of $\psi (X_{1})\times \psi (X_{2})\times \dots\times \psi (X_{\binom{n}{k-1}})$. In particular, by changing the value of one $\phi (X_{t})$
we can change $Y$ by at most $n$.
Therefore, by Lemma~\ref{Azuma}, the probability of the bad event happening is
\[
\mathbb P(Y\le \mathbb E (Y)-\mu n^k)=\mathbb P(Y \le \frac{|\mathscr P|}{r^k }|\mathcal{K}_k(G^{(k-2)})|-\mu n^k)\le \exp(-\frac{(\mu n^k)^2}{2\binom{n}{k-1}n^2})=\exp(-\Omega(n^{k-1})).
\]
In addition, there are at most $2^{n^{k-2}}<\exp(n^{k-2})$ possible choices for  $G^{(k-2)}$. 
By the union bound,  the probability of the event ``$H(n)$ is not $({|\mathscr P|}/{r^k}, \mu, k-2)$-dense" is at most $\exp(n^{k-2}-\Omega(n^{k-1}))=o(1)$. 
Therefore, for every $\mu> 0$ and sufficiently large $n$,
there exists $H\in H(n)$ is $({|\mathscr P|}/{r^k}, \mu, k-2)$-dense.
Recalling the condition given in Theorem~\ref{low-k-graphs},  $H$ is also $F$-free. 
Thus, $\pi_{k-2}(F)\ge |\mathscr P|/r^k$.
\end{proof}

\section{The hypergraph regularity method}\label{sec-regular-method}
In this section, we state the hypergraph regularity lemma and an accompanying embedding lemma. Here we follow the approach from R\"{o}dl and Schacht~\cite{regularity-lemmas}, combined with results from~\cite{count-dense, RRS-Mantel}. Meanwhile, we derive a ``clean" version of the regularity lemma for $k$-graphs (see Corollary~\ref{clean}).
The central concepts of hypergraph regularity lemma are regular complexes and equitable partition. Before we state the hypergraph regularity lemma, we introduce some necessary notation below.
For reals $x,y,z$ we write $x = y \pm z$ to denote that $y-z \leq x \leq y+z$.

\subsection{Regular complexes}
A \emph{mixed hypergraph} $\mathcal{H}$ consists of a vertex set $V(\mathcal{H})$ and an edge set $E(\mathcal{H})$, where every edge $e \in E(\mathcal{H})$ is a non-empty subset of $V(\mathcal{H})$. So a $k$-graph as defined earlier is a $k$-uniform hypergraph in which every edge has size $k$. We call a mixed hypergraph $\mathcal{H}$ a \emph{complex} if every non-empty subset of every edge of $\mathcal{H}$ is also an edge of $\mathcal{H}$.  Note that all complexes considered in this paper have the property that all vertices are contained in an edge. A complex is a \emph{$k$-complex} if its all the edges consist of at most $k$ vertices. Given a $k$-complex $\mathcal{H}$, for each $i \in [k]$, the edges of size $i$ are called \emph{$i$-edges} of $\mathcal{H}$ and we denote by $H^{(i)}$ the \emph{underlying $i$-graph} of $\mathcal{H}$: the vertices of $H^{(i)}$ are those of $\mathcal{H}$ and the edges of $H^{(i)}$ are the $i$-edges of $\mathcal{H}$. Note that every $k$-graph $H$ can be turned into a $k$-complex by making every edge into a \emph{complete $i$-graph} $K^{(i)}_k$  on $k$ vertices, for each $ i\in [k]$.

Given $i \geq 2$, let an $i$-graph $H^{(i)}$ and an $(i-1)$-graph $H^{(i-1)}$ be on the same vertex set. We define the \emph{relative density} $d(H^{(i)}|H^{(i-1)})$ of $H^{(i)}$ w.r.t. $H^{(i-1)}$ to be
\[
d(H^{(i)}|H^{(i-1)}):= \begin{cases}
	\frac{|E(H^{(i)})\cap \mathcal{K}_i(H^{(i-1)})|}{|\mathcal{K}_i(H^{(i-1)})|} &\text{if\ } |\mathcal{K}_i(H^{(i-1)})|>0,\\
	0&\text{otherwise}.
\end{cases}
\]
More generally, if $\mathbf{Q}:= (Q(1), Q(2),\dots, Q(r))$ is a collection of $r$ subhypergraphs of $H^{(i-1)}$, then we define $\mathcal{K}_i(\mathbf{Q}):= \bigcup^r_{j=1}\mathcal{K}_i(Q(j))$ and
\[
d(H^{(i)}|\mathbf{Q}):= \begin{cases}
	\frac{|E(H^{(i)})\cap \mathcal{K}_i(\mathbf{Q})|}{|\mathcal{K}_i(\mathbf{Q})|} &\text{if\ } |\mathcal{K}_i(\mathbf{Q})|>0,\\
	0&\text{otherwise}.
\end{cases}
\]

Given positive integers $s \geq k$, an \emph{$(s,k)$-graph} $H^{(k)}_s$ is an $s$-partite $k$-graph, by which we mean that the vertex set of $H^{(k)}_s$ can be partitioned into sets $V_1,\dots, V_s$ such that every edge of $H^{(k)}_s$ meets each $V_i$ in at most one vertex for $i\in [s]$. Similarly, an \emph{$(s,k)$-complex} $\mathcal{H}^{\leq k}_s$ is an $s$-partite $k$-complex.

Let integer $r\ge 1$, reals $d_i\ge 0$ and $\delta>0$ be given along with an $(i,i)$-graph $H^{(i)}_i$ and an $(i, i-1)$-graph $H^{(i-1)}_i$ on the same vertex set.
We say $H^{(i)}_i$ is \emph{$(d_i, \delta, r)$-regular}  w.r.t. $H^{(i-1)}_i$ if every $r$-tuple $\mathbf{Q}$ with $|\mathcal{K}_i(\mathbf{Q})| \geq \delta|\mathcal{K}_i(H^{(i-1)}_i)|$ satisfies $d(H^{(i)}_i|\mathbf{Q}) = d_i \pm \delta$. 
Moreover, for two $s$-partite $i$-graph $H^{(i)}_{s}$ and $(i-1)$-graph $H^{(i-1)}_{s}$  on the same vertex partition $V_1 \cup\dots \cup V_s$, we say that $H^{(i)}_{s}$ is \emph{$(d_i, \delta,r)$-regular} w.r.t. $H^{(i-1)}_{s}$ if for every $\Lambda_i\in  [s]^i$ the restriction $H^{(i)}_{s}[\Lambda_i]=H^{(i)}_{s}[\cup_{\lambda\in \Lambda_i }V_{\lambda}]$ is \emph{$(d_i, \delta,r)$-regular} w.r.t. the restriction $H^{(i-1)}_{s}[\Lambda_i]=H^{(i-1)}_{s}[\cup_{\lambda\in \Lambda_i }V_{\lambda}]$.

\begin{definition}[regular complex]
	Let integers $s \ge k \ge 3$, real $\delta>0$ and $\mathbf{d}=(d_2, \dots, ,d_{k-1})\in \mathbb{R}^{[k-2]}_{\ge 0}$. We say an $(s, k-1)$-complex $\mathcal{H}^{\le k-1}_s=\{H^{(i)}_s\}^{k-1}_{i=1}$ is \emph{$(\mathbf{d}, \delta, 1)$-regular} if $H_s^{(i)}$ is $(d_i, \delta, 1)$-regular w.r.t $H_{s}^{(i-1)}$ for every $i = 2, \dots , k-1$.
\end{definition}

%
%

\subsection{Equitable partitions}\label{section-equ-partition}
Suppose that $V$ is a finite vertex set and $\mathcal{P}^{(1)}=\{V_1, \dots , V_{a_1}\}$ is a partition of $V$, which will be called \emph{clusters}. Given $k \geq 3$ and any $j \in [k]$, we denote by $\mathrm{Cross}_j = \mathrm{Cross}_j (\mathcal{P}^{(1)})$, the family of all crossing $j$-sets $J\in [V]^k$ with $|J\cap V_i|\le 1$ for every $V_i\in \mathcal{P}^{(1)}$. 
For every  index set $\Lambda \subseteq [a_1]$ with $2 \leq |\Lambda| \leq k-1$, we write $\mathrm{Cross}_{\Lambda}$ for the family of all $ |\Lambda|$-sets of $V$ that meet each $ V_i$ with $i \in \Lambda$. Let $\mathcal{P}_{\Lambda}$ be a partition of $\mathrm{Cross}_{\Lambda}$. We refer to the partition classes of $\mathcal{P}_{\Lambda}$ as \emph{$ |\Lambda|$-cells}. For each $i = 2, \dots , k-1$, let $\mathcal{P}^{(i)}$ be the union of all the $\mathcal{P}_{\Lambda}$ with $|\Lambda| = i$. So $\mathcal{P}^{(i)}$ is a partition of $\mathrm{Cross}_i$ into several $(i,i)$-graphs.

Set $1 \leq i <j\le k$. Note that for every $i$-set $I\in \mathrm{Cross}_i$, there exists a unique $i$-cell $P^{(i)}_I\in \mathcal{P}^{(i)}$ so that $I\in P^{(i)}_I$. For every $j$-set $J\in \mathrm{Cross}_j$ we define the \emph{polyad} of $J$ as:
\[
\hat{P}^{(i)}_J:=\bigcup\big\{P^{(i)}_I: I\in [J]^{i}\big\}.
\]
So we can view $\hat{P}^{(i)}_J$ as a $(j,i)$-graph whose vertex classes are clusters intersecting $J$ and edge set is $\bigcup_{I\in [J]^{i}}E(P_I^{(i)})$. Let $\mathcal{\hat{P}}^{(j-1)}$ be the family of all polyads $\hat{P}^{(j-1)}_J$ for every $J\in \mathrm{Cross}_j$. It is easy to verify $\{\mathcal{K}_j(\hat{P}^{(j-1)}) : \hat{P}^{(j-1)}\in \mathcal{\hat{P}}^{(j-1)}\}$ is also a partition of $\mathrm{Cross}_j$.

\begin{definition}[family of partitions]\label{def-f-partition}
	Suppose $V$ is a vertex set, $k\ge 2$ is an integer and	$\mathbf{a} =(a_1,\dots, a_{k-1})$ is a vector of positive integers. We say $\bm {\mathcal P} = \bm{\mathcal P}(k-1, \mathbf{a} )= \{\mathcal{P}^{(1)}, \dots ,\mathcal{P}^{(k-1)}\}$ is a \textit{family of partitions} on $V$, if the following conditions hold:
\begin{enumerate}
		\item[$\bullet$]  $\mathcal{P}^{(1)}$ is a partition of $V$ into $a_1$ clusters.
		
		\item[$\bullet$] $\mathcal{P}^{(i)}$ is a partition of $\mathrm{Cross}_i$ satisfying 
		\begin{equation}\label{cells}
			|\{P^{(i)}\in \mathcal{P}^{(i)}: P^{(i)}\subseteq \mathcal{K}_i(\hat{P}^{(i-1)})\}|=a_i	
		\end{equation}
		for every $\hat{P}^{(i-1)}\in \mathcal{\hat{P}}^{(i-1)}$.	
	\end{enumerate}  
\end{definition}	

So for each $J \in \mathrm{Cross}_j$ we can view $\bigcup^{j-1}_{i=1}\hat{P}^{(i)}_J$ as a $(j, j-1)$-complex.

\begin{definition}[$(\eta, \delta, t)$-equitable]\label{def-eq-partition}
	Suppose $V$ is a set of $n$ vertices, $t\in \mathbb N$, $\mathbf{a} =(a_1,\dots, a_{k-1})\in \mathbb N^{[k-1]}$  and $\eta, 
	\delta>0$. We say a family of partitions  $\bm{\mathcal P} = \bm{\mathcal P}(k-1,\mathbf{a} )$ is \emph{$(\eta, \delta, t)$-equitable} if it satisfies the following:
\stepcounter{propcounter}
\begin{enumerate}[label = ({\bfseries \Alph{propcounter}\arabic{enumi}})]
\item\label{p1} $\mathcal{P}^{(1)}$ is a partition of $V$ into $a_1$ clusters of equal size, where $1/\eta \leq a_1 \leq t$ and $a_1$ divides $n$.
\item\label{p2} $\mathcal{P}^{(i)}$ is a partition of $\mathrm{Cross}_i$ into at most $t$ for $i=2, \dots , k-1$.
\item\label{p3} For every $k$-set $K \in \mathrm{Cross}_k$, the $(k, k-1)$-complex $\bigcup^{k-1}_{i=1}\hat{P}^{(i)}_K$ is $(\mathbf{d}, \delta,1)$-regular, where  $\mathbf{d}= (1/a_2, \dots , 1/a_{k-1})$. 
\item\label{p4}  For every $j\in [k-1] $ and every $k$-set $K \in \mathrm{Cross}_k$, we have
\[
|\mathcal{K}_k(\hat{P}^{(j)}_K)|= (1\pm \eta) \prod \limits_{\ell=1}^j (\frac{1}{a_{\ell}})^{\binom{k}{\ell}} n^k.
\]
\end{enumerate}
\end{definition}

\begin{remark} 
The condition~\ref{p3} of Definition~\ref{def-eq-partition} implies that the $i$-cells of $\mathcal{P}^{(i)}$ have almost equal size, and condition~\ref{p4} of Definition~\ref{def-eq-partition} is not a part of the statement of $(\eta, \delta, t)$-equitable from R\"{o}dl and Schacht~\cite{regularity-lemmas}.
The condition~\ref{p4}  is actually a consequence of conditions~\ref{p1} and~\ref{p3} and the so-called dense counting lemma from~\cite[Theorem 6.5]{count-dense} (see also~\cite[Theorem 3.1]{regularity-lemmas} or~\cite[Theorem 2.1]{counting-lemmas}).
\end{remark}

\subsection{Statements of the regularity lemma and embedding lemma }
Suppose $\delta_k$ is a positive real and $r$ is a positive
integer. Let $H$ be a $k$-graph on $V$ and $\bm{\mathcal P} = \bm{\mathcal P}(k-1,\mathbf{a} )$  is a family of partitions on $V$. Given a polyad $\hat{P}^{(k-1)} \in \hat{\mathcal{P}}^{(k-1)}$, we say that $H$ is \emph{$(\delta_k, r)$-regular} w.r.t. $\hat{P}^{(k-1)}$ if $H$ is $(d_k, \delta_k, r)$-regular w.r.t. $\hat{P}^{(k-1)}$ where $d_k=d(H|\hat{P}^{(k-1)})$. Finally, we define that $H$ is \emph{$(\delta_k, r)$-regular} w.r.t. $\bm{\mathcal P}$.

\begin{definition}[\emph{$(\delta_k, r)$-regular} w.r.t. $\bm{\mathcal P}$]
	We say a $k$-graph $H=(V,E)$ is \emph{$(\delta_k, r)$-regular} w.r.t. $\bm{\mathcal P}$ if
	\[
	\big|\bigcup\big\{\mathcal{K}_k(\hat{P}^{(k-1)}) : \hat{P}^{(k-1)}\in \mathcal{\hat{P}}^{(k-1)}\\
	\text{and\ } H \text{\ is\ not\ } (\delta_k, r)\text{-regular\ w.r.t.\ } \hat{P}^{(k-1)} \big\}\big| \le \delta_k |\mathrm{Cross}_k|.
	\]
\end{definition}
This means that no more than a $\delta_k$-fraction of the $k$-sets of $V$ form a $K_k^{(k-1)}$ that lies within a polyad w.r.t. which $H$ is not regular.

Now we are ready to state the regularity lemma for $k$-graphs.
\begin{theorem}[Regularity lemma~{\cite[Theorem 2.3]{regularity-lemmas}}] \label{thm-Reg-lem}
	Let $k\geq 2$ be a fixed integer. For all positive constants $\eta$ and $\delta_k$ and all functions $r:\mathbb{N}^{[k-1]} \rightarrow \mathbb{N}$ and $\delta:\mathbb{N}^{[k-1]} \rightarrow (0,1]$, there are integers $t$ and $n_0$ such that the following holds.
	For every $k$-graph $H$ of order $n\ge n_0$ and $t!$ dividing $n$, there exists a family of partitions $\bm{\mathcal P} = \bm{\mathcal P}(k-1,\mathbf{a} )$  of $V(H)$ with $\mathbf{a}=(a_1,\dots, a_{k-1})\in \mathbb{N}^{[k-1]}$  such that
	\begin{enumerate}
		\item[$(1)$] $\bm{\mathcal P}$ is $(\eta, \delta(\mathbf{a}), t)$-equitable and
		
		\item[$(2)$] $H$ is $(\delta_k, r(\mathbf{a}))$-regular w.r.t. $\bm{\mathcal P}$.
	\end{enumerate}	
\end{theorem}

%

Similar to in other proofs based on the regularity method it will be convenient to ``clean" the family of partitions provided by Theorem~\ref{thm-Reg-lem}.
Given a finite set $V$ and a family of partitions $\bm{\mathcal P} = \bm{\mathcal P}(k-1,\mathbf{a} )$ on $V$
with $\mathbf{a}=(a_1,\dots, a_{k-1})$ and  $k\ge 3$,  
we call an $a_1$-set $T\subset V$ a \emph{transversal} of $\mathcal P^{(1)}$ if $T$ satisfies $|T \cap V_i| = 1$ for every $i\in [a_1]$.
Given a transversal $T$ of $\mathcal P^{(1)}$, we consider the selection 
\[
\mathcal G_T=\{P^{(k-2)}_J\in \mathcal P^{(k-2)}: J\in [T]^{k-2}\}
\]
and let 
\[
\mathcal K_k(\mathcal G_T)=\{K\in \mathrm{Cross}_k:P^{(k-2)}_J\in \mathcal G_T ~\text{ for~every~} J\in [K]^{k-2} \}
\]
be the collection of $k$-sets of $V$ that are supported by  $\mathcal G_T$.

\begin{corollary}\label{clean}
	Let $m\ge k\geq 3$ be fixed integers. For all positive constants $\eta\ll m^{-1}$ and $\delta_k<d_k$ and all functions $r:\mathbb{N}^{[k-1]} \rightarrow \mathbb{N}$ and $\delta:\mathbb{N}^{[k-1]} \rightarrow (0,1]$, there are integers $t$ and $n_0$ such that the following holds.
	For every $k$-graph $H=(V,E)$ of order $n\ge n_0$ and $t!$ dividing $n$, there exists a  subhypergraph $\hat{H}=(\hat{V},\hat{E})$ of $H$,
	and  a family of partitions $\bm{\mathcal P} = \bm{\mathcal P}(k-1,\mathbf{a} )$  of $\hat{V}$ with $\mathbf{a}=(m, a_2, \dots, a_{k-1})\in \mathbb{N}^{[k-1]}$ satisfying the following properties:
\stepcounter{propcounter}
\begin{enumerate}[label = {\rm ({\bfseries \Alph{propcounter}\arabic{enumi}})}]
	\item\label{p11} $\bm{\mathcal P}$ is $(\eta, \delta(\mathbf{a}), t)$-equitable.
	\item\label{p22} For every $k$-set $K\in {\rm Cross}_k$,  $\hat H$ is $(\delta_k, r)$-regular w.r.t.   $\hat{P}_K^{(k-1)}$, and $d(\hat H|\hat{P}_K^{(k-1)})$ is either 0 or at least $d_k$.
	\item\label{p33}  There is a transversal $T$ of $\mathcal P^{(1)}$ such that 	for each $\mathcal Y\in [m]^k$
	\[
	|\mathcal K_k(\mathcal G_T) \cap \mathrm{Cross}_{\mathcal Y}\cap \hat E|\ge |\mathcal K_k(\mathcal G_T) \cap \mathrm{Cross}_{\mathcal Y}\cap E|-2d_k|\mathcal K_k(\mathcal G_T) \cap \mathrm{Cross}_{\mathcal Y}|.
	\]
\end{enumerate}
\end{corollary}

\begin{proof} Suppose that we have constants
	\[
	{n_0}^{-1}\ll r^{-1}, \delta\ll \min\{\delta_k, a_1^{-1}, \dots, a_{k-1}^{-1}, t^{-1}\} \ll \delta_k, \eta\ll d_k, m^{-1}\le k^{-1}.
	\]
	We shall apply the regularity lemma (Theorem~\ref{thm-Reg-lem}) with $\eta$, $\delta_k$ sufficiently small 
	and functions $r:\mathbb{N}^{[k-1]} \rightarrow \mathbb{N}$ and $\delta:\mathbb{N}^{[k-1]} \rightarrow (0,1]$, thus receiving two large integers $t$ and $n_0$. Let $H=(V,E)$ be a $k$-graph of order $n\ge n_0$  and $t!$ dividing $n$.
	We apply Theorem~\ref{thm-Reg-lem} to $H$ to obtain a family of partitions $\bm{\mathcal P'} = \bm{\mathcal P'}(k-1,\mathbf{a'} )$  of $V$ with $\mathbf{a'}=(a_1,\dots, a_{k-1})\in \mathbb{N}^{[k-1]}_{>0}$  such that
	\[
	\bm{\mathcal P'} ~\text{is}~(\eta, \delta(\mathbf{a'}), t)\text{-equitable and}~ H ~\text{is}~(\delta_k, r(\mathbf{a'}))\text{-regular w.r.t.}~\bm{\mathcal P}'. 
	\]
	Given a transversal $T$ of $\mathcal P^{(1)}$, we have
	\[
	\mathcal G_T=\{P^{(k-2)}_J\in \mathcal P^{(k-2)}: J\in [T]^{k-2}\}.
	\]
	Since $\bm{\mathcal P}'$ is $(\eta, \delta, t)$-equitable, recalling  the property~\ref{p4} in  Definition~\ref{def-eq-partition},   for each transversal $T$ of $\mathcal P^{(1)}$ and every $k$-set $K\in [T]^k$, we have
	\[
	|\mathcal{K}_k(\hat{P}^{(k-2)}_K)|= (1\pm \eta) \prod \limits_{\ell=1}^{k-2} (\frac{1}{a_{\ell}})^{\binom{k}{\ell}} n^k.
	\]
	Therefore,  every polyad $\hat{P}^{(k-2)}_K$  has the same volume up to a multiplicative factor controlled by $\eta$.
	In addition, since $H$ is $(\delta_k, r)$-regular w.r.t. $\bm{\mathcal P}'$,  there are all but at most $\delta_k|\mathrm{Cross}_k|$ $k$-sets $K$ in  $\mathrm{Cross}_k$ having the property that $H$ is $(\delta_k,r)$-regular w.r.t. $\hat P_K^{(k-1)}$.
	An easy averaging argument shows that there are some appropriate transversal $T$ such that all but at most $2\delta_k |\mathcal K_k(\mathcal G_T)|$ members of $\mathcal K_k(\mathcal G_T)$ have the property that $H$ is $(\delta_k,r)$-regular w.r.t. their polyads. 
	From now on we fix one
	such choice of $T$ and the corresponding collection $\mathcal G_T$. 
	
	For each $\mathcal Y \in [a_1]^{k}$, recall that $
	\mathrm{Cross}_{\mathcal Y}=\{K\in\mathrm{Cross}_k: K\cap V_i\neq \emptyset \text{~for~} i\in \mathcal Y\}$.
	Now we consider an auxiliary $k$-graph $R=([a_1], E_R)$ on the vertex set $[a_1]$, where
$\mathcal Y\in E_R$ if $\mathcal Y$ satisfies the following property:
\[
|\{ K\in  \mathcal K_k(\mathcal G_T) \cap \mathrm{Cross}_{\mathcal Y}: H \textit{~is~not~} (\delta_k. r)\textit{-regular~w.r.t.~} \hat{P}^{(k-1)}_K\}|>2\sqrt{\delta_k} |\mathcal K_k(\mathcal G_T) \cap \mathrm{Cross}_{\mathcal Y}|.
	\]
	By the choice of $\mathcal G_T$, we can obtain that 
	\[
	|E_R|\le \frac{2\delta_k |\mathcal K_k(\mathcal G_T)|}{2\sqrt{\delta_k} |\mathcal K_k(\mathcal G_T) \cap \mathrm{Cross}_{\mathcal Y}|}\le 
	2\sqrt{\delta_k} \binom{a_1}{k}. 	
	\]
	Consequently, owing to the choice of  $\delta_k \ll  1/m$ and $m\ll \eta\le a_1$, the auxiliary $k$-graph $R$ has an independent set $M\subseteq [a_1]$ of size $m$.
	
	Finally, we construct the desired subhypergraph $\hat{H}=(\hat V, \hat E)$. Let $\hat V:= \cup_{\lambda\in M }V_{\lambda} $ and $\bm{\mathcal P} = \bm{\mathcal P}(k-1,\mathbf{a} )$ be  
	the family of partitions $\bm{\mathcal P}'$ restricted under set $M$. Clearly, 
	$\mathbf{a}=(m,a_2, \dots, a_{k-1})$.
	Let us remove the edges from $\mathcal K_k(\mathcal G_T) \cap E$ which lie in a polyad $\hat P^{(k-1)}$ such that $\hat H$ is not $(\delta_k, r)$-regular 
	w.r.t. $\hat P^{(k-1)}$.
	By the choice of $M$ and $\delta_k\ll d_k$, for each $\mathcal Y\in [M]^k$, the number of edges we removed from $\mathcal K_k(\mathcal G_T) \cap E$ is at most $2\sqrt{\delta_k} |\mathcal K_k(\mathcal G_T) \cap \mathrm{Cross}_{\mathcal Y}|< d_k|\mathcal K_k(\mathcal G_T) \cap \mathrm{Cross}_{\mathcal Y}|$. 
	Moreover, we also remove the edges from $\mathcal K_k(\mathcal G_T) \cap E$ which lie in a polyad $\hat P^{(k-1)}$ such that $d(H|\hat P^{(k-1)})<d_k$. 
	Let $\hat E$ be the resulting edge set after these deletions. Then for each $\mathcal Y\in [M]^k$ we have
	\[
	|\mathcal K_k(\mathcal G_T) \cap \mathrm{Cross}_{\mathcal Y}\cap \hat E|\ge |\mathcal K_k(\mathcal G_T) \cap \mathrm{Cross}_{\mathcal Y}\cap E|-2d_k|\mathcal K_k(\mathcal G_T) \cap \mathrm{Cross}_{\mathcal Y}|.
	\]
	Therefore, $\hat{H}$ has all the desired properties.
\end{proof}

Finally, we state a general  embedding lemma, which allows embedding $k$-graphs of fixed isomorphism type into appropriate and sufficiently regular and dense polyads of the partition provided by Corollary~\ref{clean}. It is a direct consequence of~\cite[Theorem 2]{Embeddings}.

\begin{theorem}[Embedding lemma] \label{E-lemma}
	Let $f,k,r,n_0$ be positive integers and
	let $\mathbf{d}=(d_2,\dots, d_{k-1})\in \mathbb{N}^{[k-2]}_{>0}$ such that $1/d_i\in \mathbb{N}$ for all $i < k$,
	\[
	n_0^{-1}\ll r^{-1}, \delta \ll \min\{\delta_k,d_2,\dots, d_{k-1}\}\le \delta_k \ll d_k, 1/f.
	\]
	Then the following holds for all integers $n\ge n_0$. Let $F$ be a $k$-graph with vertex set $[f]$.  Suppose that $\mathcal{H}=\{H^{(j)}\}^{k-1}_{j=1}$ is a $(\mathbf{d}, \delta, 1)$-regular $(f, k-1)$-complex with clusters $V_1, \dots, V_f$, all of size $n$. Suppose also that $H$ is an $f$-partite $k$-graph on the same vertex partition such that for each edge  $\{i_1,\dots, {i_k}\}\in E(F)$, $H$ is $(\delta_k, r)$-regular w.r.t. the restriction $H^{(k-1)}[V_{i_1}\cup \dots \cup V_{i_k}]$ and  $d(H|H^{(k-1)}[V_{i_1}\cup \dots \cup V_{i_k}])\ge d_k$.
	Then $H$ contains a copy of $F$.
\end{theorem}

\section{Proof of Theorem~\ref{thm-turan-reduced}}\label{sec-Reduced-pf}
For the proof of Theorem~\ref{thm-turan-reduced}, we intend to apply Theorem~\ref{E-lemma} (embedding lemma).
To apply Theorem~\ref{E-lemma}, we need to keep track of which polyads are dense and regular.
Similar to the role of reduced graphs in Szemer\'edi's regularity method, 
we hope that reduced $k$-graphs $\mathcal A$ is well suited for analyzing the  structure of the partition provided by Corollary~\ref{clean} applied to a host $k$-graph $H$. 
In other words, we hope that $(d+\eps)$-dense reduced $k$-graphs  $\mathcal A$ can inherit some useful properties of $(d+\eps',\mu, k-2)$-dense $k$-graphs $H$ for $0<\eps<\eps'\ll 1$.

For the above purposes it will be more convenient to work with an alternative definition of $\pi_{j}(F)$ that
we denote by $\pi_{[k]^{j}}(F)$ from Reiher, R\"odl and Schacht~\cite{RRS-Mantel}. 
In contrast to Definition~\ref{def-j-dense}, it speaks about the edge distribution
of $H$ relative to families consisting of $\binom{k}{k-j}$
many $j$-graphs rather
than just relative to one such $j$-graph.

Given a finite set $V$ and integer $k\ge 3$, we identify the Cartesian power $V^{[k]}$ by regarding any $k$-tuple $\vec{v} = (v_1,\dots ,v_k)$ as being the function $i\mapsto v_i$. Furthermore, 
for a set $J \in  [k]^j$ with $j<k$, we write $V^J$ for the set of all functions from $J$ to $V$. 
In this way, the natural projection from $V^{[k]}$ to $V^S$ becomes the restriction $\vec{v}\mapsto \vec{v}\mid S$ and the preimage of any set $G_S\subseteq V^S$ is denoted by
\[
\mathcal{K}_k(G_S) = \{\vec{v}\in V^{[k]}: (\vec{v}\mid S)\in G_S \}.
\]
More generally, for a family $\mathscr{G}_j = \{G_J: J\in [k]^j\}$ with $G_J \subseteq V^J$ for all $J\in [k]^j$, let
\[
\mathcal{K}_k(\mathscr{G}_j)= \bigcap \limits_{J\in [k]^j} \mathcal{K}_k(G_J).
\]
Given a $k$-graph $H = (V,E)$, let
\[
e_H(\mathscr{G}_j)=|\big \{(v_1,\dots ,v_k)\in \mathcal{K}_k(\mathscr{G}_j):\{ v_1,\dots ,v_k\}\in E\big \}|.
\]

\begin{definition}[\cite{RRS-Mantel}] \label{k-j-dense}
	Given integers $n\ge k>j\ge 0$, let real numbers $d\in [0,1]$, $\mu>0$, and $H = (V, E)$ be a
	$k$-graph with $n$ vertices. We say that $H$ is {\it $(d, \mu, [k]^j)$-dense } if
	\begin{equation}\label{equ-k-j-dense}
		e_H(\mathscr{G}_j)\geq d|\mathcal{K}_k(\mathscr{G}_j)|-\mu n^k
	\end{equation}
	holds for every family $\mathscr{G}_j=\{G_J: J\in [k]^j\}$ associating with each $J\in [k]^j$ some $G_J\subseteq V^J$.
\end{definition}

Accordingly, we set
\begin{equation*}
	\begin{split}
		\label{k-j-turan-dense}
		\pi_{[k]^j}(F) = \sup \{ d\in [0,1] & : \text{for\ every\ } \mu>0 \ \text{and\ } n_0\in \mathbb{N},\ \text{there\ exists\ an\ } F \text{-free} \\
		&\quad (d,\mu,[k]^j)\text{-dense}~ k \text{-graph~} H\ \text{with~} |V(H)|\geq n_0 \}.
	\end{split}
\end{equation*}

Reiher, R\"odl and Schacht~\cite[Proposition 2.5]{RRS-Mantel} proved the following result.

\begin{proposition}\label{pro-turan-trans}
	For positive intrgers $k>j>0$, every $k$-graph $F$ satisfies	
	\begin{equation*}
		\pi_j(F)=\pi_{[k]^j}(F).
	\end{equation*}
\end{proposition}

Consequently it is allowed to imagine that in Theorem~\ref{1/kthm} we would have
written $\pi_{[k]^{k-2}}(F)$ instead of $\pi_{k-2}(F)$. 
Now we can transform the embedding problems in $(d, \mu, {k-2})$-dense $k$-graphs into the embedding problems in $(d, \mu, [k]^{k-2})$-dense $k$-graphs, and give the proof of Theorem~\ref{thm-turan-reduced} using Corollary~\ref{clean} and Theorem~\ref{E-lemma}.

\begin{proof}[Proof Theorem~\ref{thm-turan-reduced}] Given $k\ge 3$, $d\in [0,1]$ and $\eps>0$, we choose $m^{-1}\ll \eps$.
Suppose that  $F$ is a $k$-graph  satisfying the statement of
Theorem~\ref{thm-turan-reduced} and $|V(F)|=m$. We fix auxiliary constants and functions to satisfy the hierarchy
	\[
	0<\mu\ll {n_0}^{-1}\ll r(\cdot)^{-1}, \delta(\cdot) \ll \min\{\delta_k, a_2^{-1}, \dots, a_{k-1}^{-1}, t^{-1}\} \ll \delta_k, \eta\ll d_k, m^{-1},
	\]	
	where $\delta_k$ and the
	functions $r(\cdot)$ and $\delta(\cdot)$ are given by Theorem~\ref{E-lemma} applied for $F$ and  $d_k$, and $\eta, t$ are given by  Corollary~\ref{clean}.
	By Proposition~\ref{pro-turan-trans},
	it suffices to show that $\pi_{[k]^{k-2}}(F)\le d$.

	Let  $H$ be a $(d+2\eps, \mu, [k]^{k-2})$-dense $k$-graph on $n\ge n_0$ vertices. 
	By Corollary~\ref{clean} applied to $H$, we obtain a subhypergraph $\hat H=(\hat V, \hat E)$ of $H$ and a family of partitions $\bm{\mathcal P} = \bm{\mathcal P}(k-1,\mathbf{a} )$  of $\hat{V}$ with $\mathbf{a}=(m, a_2, \dots, a_{k-1})\in \mathbb{N}^{k-1}_{>0}$  satisfying properties~\ref{p11}-\ref{p33} of Corollary~\ref{clean}. Set $\mathcal P^{(1)}=\{V_1, V_2, \dots, V_m\}$. Recalling the property~\ref{p33} of Corollary~\ref{clean}, there is a transversal $T$ of $\mathcal P^{(1)}$ such that 	for each $\mathcal Y\in [m]^k$
	\begin{equation}\label{equ-thm-31-0}
		|\mathcal K_k(\mathcal G_T) \cap \mathrm{Cross}_{\mathcal Y}\cap \hat E|\ge |\mathcal K_k(\mathcal G_T) \cap \mathrm{Cross}_{\mathcal Y}\cap E|-2d_k|\mathcal K_k(\mathcal G_T) \cap \mathrm{Cross}_{\mathcal Y}|.
	\end{equation}
	
Now we construct an $m$-reduced $k$-graph $\mathcal A$  with index set $[m]$ as
follows:  For
	each $\mathcal X\in [m]^{k-1}$, the vertex class $\mathcal P_{\mathcal X}$ is defined to be the set of all $k-1$-cells $P^{k-1}\in \mathcal P^{k-1}$ with $P^{k-1}\in \mathcal K_{k-1}(\hat P^{(k-2)}_{T_{\mathcal X}})$ where $T_{\mathcal X}:=\{T\cap V_i: i\in \mathcal X \}$ and $\hat P^{(k-2)}_{T_{\mathcal X}}=\bigcup \{P^{(k-2)}_I: I\in [T_{\mathcal X}]^{k-2}\}$.
	As a consequence all the vertex classes $\mathcal P_{\mathcal X}$
	have the same size $a_{k-1}$ since $\bm{\mathcal P}$ is a family of partitions, see equation~(\ref{cells}). 
	It remains to define the constituents of $\mathcal A$. For simplicity, let $P^{(k-1)}(w)$ denote the $(k-1)$-cell corresponding to  $w\in \mathcal P_{\mathcal X}$.
	Given a $k$-set $\mathcal Y\in [M]^k$, we let $E(\mathcal A_{\mathcal Y})$ be the collection of all
	$k$-sets $\{w_1, w_2, \dots, w_k\}$ of $\bigcup_{\mathcal  X\in [\mathcal  Y]^{k-1}}\mathcal P_{\mathcal  X}$ such that $\bigcup \{P^{(k-1)}(w_i): i\in [k]\}$ forms a
	$k$-partite $(k-1)$-graph $\hat P^{(k-1)}$ (polyad) w.r.t. which $\hat H$ is $(\delta_k,r)$-regular and $d(\hat H| \hat P^{(k-1)})\ge d_k$.

We first claim that the $m$-reduced $k$-graph $\mathcal A$ is $(d+\eps)$-dense.
	Given a $k$-set $\mathcal Y\in [m]^k$, since $H$ is $(d+2\eps, \mu, [k]^{k-2})$-dense, we have that
	\begin{equation}\label{equ-thm-31-1}
		|\mathcal K_k(\mathcal G_T) \cap \mathrm{Cross}_{\mathcal Y}\cap E|\ge (d+2\eps)|\mathcal K_k(\mathcal G_T) \cap \mathrm{Cross}_{\mathcal Y}|-\mu n^k.
	\end{equation}
	Note that  $\mathcal K_k(\mathcal G_T) \cap \mathrm{Cross}_{\mathcal Y}=\mathcal K_k(\hat P^{(k-2)}_{T_{\mathcal Y}})$
	where $T_{\mathcal Y}=\{T\cap V_i: i\in \mathcal Y \}$ and $\hat P^{(k-2)}_{T_{\mathcal Y}}=\bigcup \{P^{(k-2)}_I: I\in [T_{\mathcal Y}]^{k-2}\}$.
	Since $\bm{\mathcal P'}$ is $(\eta, \delta(\mathbf{a}), t)$-equitable, by the condition~\ref{p4} of Definition~\ref{def-eq-partition}, we have 
	\begin{equation}\label{equ-thm-31-2}
		|\mathcal K_k(\mathcal G_T) \cap \mathrm{Cross}_{\mathcal Y}|= (1\pm \eta) \prod \limits_{\ell=1}^{k-2} (\frac{1}{a_{\ell}})^{\binom{k}{\ell}} n^k,
	\end{equation}
	and every polyad $\hat P^{(k-1)}$ satisfies
	\begin{equation}\label{equ-thm-31-3}
		|\mathcal K_k(\hat P^{(k-1)}) |= (1\pm \eta) \prod \limits_{\ell=1}^{k-1} (\frac{1}{a_{\ell}})^{\binom{k}{\ell}} n^k,
	\end{equation}
	Combining the lower bound in (\ref{equ-thm-31-2}) with our choice $\mu\ll t^{-1}$, $\eps$ leads to 
	\begin{equation*}
		|\mathcal K_k(\mathcal G_T) \cap \mathrm{Cross}_{\mathcal Y}|\ge (1- \eta) \cdot (\frac{1}{t})^{\sum_{\ell=1}^{k-2}\binom{k}{\ell}} n^k\ge \frac{1}{t^{2^k}}n^k\ge \frac{2\mu}{\eps}n^k,
	\end{equation*}
	and hence (\ref{equ-thm-31-1}) can rewrite as
	\begin{equation}\label{equ-thm-31-4}
		|\mathcal K_k(\mathcal G_T) \cap \mathrm{Cross}_{\mathcal Y}\cap E|\ge (d+\frac{3}{2}\eps)|\mathcal K_k(\mathcal G_T) \cap \mathrm{Cross}_{\mathcal Y}|.
	\end{equation}
	Owing to (\ref{equ-thm-31-0}) and (\ref{equ-thm-31-4}), and $d_k\ll \eps$, we obtain
	\begin{equation}\label{equ-thm-31-5}
		|\mathcal K_k(\mathcal G_T) \cap \mathrm{Cross}_{\mathcal Y}\cap \hat E|\ge (d+\frac{5}{4}\eps)|\mathcal K_k(\mathcal G_T) \cap \mathrm{Cross}_{\mathcal Y}|.
	\end{equation}
	In particular, these edges from $\mathcal K_k(\mathcal G_T) \cap \mathrm{Cross}_{\mathcal Y}\cap \hat E$ all lie in polyads $\hat P^{(k-1)}$ that are encoded as edges of $\mathcal A_{\mathcal Y}$. However, by (\ref{equ-thm-31-3}), every polyad $\hat P^{(k-1)}$ can
	support at most
	\[
	|\mathcal K_k(\hat P^{(k-1)}) |\le  (1+ \eta)\frac{1}{a^k_{k-1}} \prod \limits_{\ell=1}^{k-2} (\frac{1}{a_{\ell}})^{\binom{k}{\ell}} n^k\le \frac{1+ \eta}{1- \eta} \cdot \frac{1}{a^k_{k-1}} \cdot |\mathcal K_k(\mathcal G_T) \cap \mathrm{Cross}_{\mathcal Y}|
	\]
	edge of $\hat H$.
	For these reasons~\eqref{equ-thm-31-5} leads to
	\[
	(d+\frac{5}{4}\eps)|\mathcal K_k(\mathcal G_T) \cap \mathrm{Cross}_{\mathcal Y}|\le |E(\mathcal A_{\mathcal Y})|\cdot \frac{1+ \eta}{1- \eta} \cdot \frac{1}{a^k_{k-1}} \cdot |\mathcal K_k(\mathcal G_T) \cap \mathrm{Cross}_{\mathcal Y}|
	\]
	which yields
	\[
	|E(\mathcal A_{\mathcal Y})|\ge (d+\eps) a^k_{k-1}=(d+\eps)  \cdot \prod_{{\mathcal X}\in [\mathcal Y]^{k-1}}|\mathcal {P}_{\mathcal X}|.
	\]
	Therefore, $\mathcal A$ is a $(d+\eps)$-dense $m$-reduced $k$-graph.

Next, let $\mathcal{H}=\{H^{(j)}\}^{k-1}_{j=1}$ denote the $(m, k-1)$-complex formed by all $(k, k-1)$-complexes
$\bigcup^{k-1}_{i=1}\hat{P}^{(i)}_K$ with $K\in \mathcal K_k(\mathcal G_T) \cap \mathrm{Cross}_{\mathcal Y}$.
Since $\bm{\mathcal P'}$ is $(\eta, \delta(\mathbf{a}), t)$-equitable, by the condition~\ref{p3} of Definition~\ref{def-eq-partition}, 
  $\mathcal{H}$ 
is $(\mathbf{d}, \delta, 1)$-regular with  $\mathbf{d}= (1/a_2, \dots , 1/a_{k-1})$.
Moreover, $\mathcal A$  embeds $F$ which means that there is a reduced map $(\phi, \psi)$ from $F$ to $\mathcal A$, which means that
if the $i_1$\textsuperscript{th}, $i_2$\textsuperscript{th}, $\dots$, $i_k$\textsuperscript{th}
vertices of $F$ form a $k$-edge, then
 \[
\{\psi(i_2,\dots, i_k), \psi(i_1,i_3,\dots, i_k),\dots, \psi(i_1,\dots, i_{k-1})\}\in E( \mathcal A_{\mathcal Y}).
\]
By the construction of $\mathcal A$, $\hat H$ is $(\delta_k, r)$-regular w.r.t. the restriction $H^{(k-1)}[V_{\phi({i_1})}\cup \dots \cup V_{\phi({i_k})}]$ and  $d(H|H^{(k-1)}[V_{\phi({i_1})}\cup \dots \cup V_{\phi({i_k})}])\ge d_k$.
Therefore, applying Theorem~\ref{E-lemma} to $\hat H$ and $F$, we have $F\subset \hat H\subset H$.
\end{proof}

\section{Embedding lemma of reduced $k$-graphs}\label{sec-Reduced-embed}

In this section, we will prove some auxiliary
results for reduced $k$-graphs and use them to prove an embedding lemma (Lemma~\ref{lem-Embedding-reduced}) of reduced $k$-graphs with density more than $k^{-k}$, which is the main result of
this section.


\begin{lemma}\label{lem-Embedding-reduced}
	Given  $\eps>0$ and integers $ m> k \ge 3$, there exists  $N\in \mathbb N$ such that the following holds.
	For every $(k^{-k}+\eps)$-dense $N$-reduced $k$-graph $\mathcal A$, there exists an induced subhypergraph $\mathcal A'\subseteq \mathcal A$  on index set $M\subseteq [N]$ with $|M|=m$ and there exist $(2k-1)$  vertices (not necessarily distinct) $\alpha^1_{\mathcal X}, \dots, \alpha^k_{\mathcal X}, \beta^1_{\mathcal X}, \dots, \beta^{k-1}_{\mathcal X} \in \mathcal P_{\mathcal X}$ for  all  ${\mathcal X}\in [M]^{k-1}$ such that the followings hold.
\stepcounter{propcounter}
\begin{enumerate}[label = {\rm ({\bfseries \Alph{propcounter}\arabic{enumi}})}]
\item\label{p111} For all $\mathcal Y\in [M]^k$ and $\mathcal X_{\ell}\in  [\mathcal Y]^{k-1}$ with $\ell \in [k]$, we have  $\{\alpha^1_{\mathcal X_{1}}, \alpha^2_{\mathcal X_{2}}, \dots, \alpha^k_{\mathcal X_{k}} \} \in E(\mathcal A_{\mathcal Y})$.
\item\label{p222} There exists a pair $\{i',j'\}\in [k]^2$ with $i'<j'$ such that for all $\mathcal Y\in [M]^k$ and $\mathcal X_{\ell}\in  [\mathcal Y]^{k-1}$ with $\ell \in [k]$, we have 
$\{\beta^1_{\mathcal X_{1}},\dots, \beta^{j'-1}_{\mathcal X_{j'-1}},  \alpha^{i'}_{\mathcal X_{j'}}, \beta^{j'}_{\mathcal X_{j'+1}}, \dots, \beta^{k-1}_{\mathcal X_{k}} \} \in E(\mathcal A_{\mathcal Y})$.
\end{enumerate}
	
\end{lemma}

We postpone the proof of Lemma~\ref{lem-Embedding-reduced} to the end of this section.
Combining Theorem~\ref{thm-turan-reduced} and Lemma~\ref{lem-Embedding-reduced}, we first give the proof of  Theorem~\ref{1/kthm}. 

\begin{proof}[Proof of Theorem~\ref{1/kthm}]
	Given $m> k\ge 3$, let $F$ be an $m$-vertex  $k$-graph  obeying conditions ($\clubsuit$) and ($\spadesuit$) in Theorem~\ref{1/kthm}.
	Recalling the condition ($\clubsuit$), $F$  has no vanishing ordering of $V(F)$. By
	 Theorem~\ref{thm-k-2-zero-density} and Corollary~\ref{lower-bound}, we trivially have
	$\pi_{k-2}(F)\ge k^{-k}$.

	Next, we shall apply Theorem~\ref{thm-turan-reduced} to prove that  $\pi_{k-2}(F)\le k^{-k}$. 
	It suffices to show that for every $\eps>0$, there exists $N\in \mathbb N$ such that every $(k^{-k}+\eps)$-dense $N$-reduced $k$-graph embeds $F$. 
We first apply  Lemma~\ref{lem-Embedding-reduced} with $\eps$ and $m$ to get $N$. Let $\mathcal A$  be a $(k^{-k}+\eps)$-dense $N$-reduced $k$-graph with index set $[N]$, and $\mathcal A'\subseteq \mathcal A$ be an  induced subhypergraph satisfying the 
properties~\ref{p111} and~\ref{p222} in Lemma~\ref{lem-Embedding-reduced} on index set $M\subseteq[N]$ with $|M|=m$.
	Now we fix $\{i',j'\}\in [k]^2$ with $i'<j'$ by the property~\ref{p222}.
	Recalling the condition ($\spadesuit$),  $F$ can be partitioned into two spanning subhypergraphs $F^1_{i',j'}$ and $F^2_{i',j'}$ such that there exists an ordering $\sigma=(v_1, v_2,\dots, v_m) $ of $V(F)$ that is vanishing both for $F^1_{i',j'}$ and $F^2_{i',j'}$
	and for any two edges $e_1\in E(F^1_{i',j'})$, $e_2\in E(F^2_{i',j'})$ with $|e_1\cap e_2|=k-1$,  $e_1\cap e_2$ is $i'$-type w.r.t. $F^1_{i',j'}$ and $j'$-type w.r.t. $F^2_{i',j'}$.
Therefore, for each $(k-1)$-set $S\in \partial F$, $S$ only satisfies one of the following three cases: 
\begin{enumerate}
\item $S$ is $r$-type w.r.t. $F^1_{i',j'}$ for some $r\in [k]$; 
\item $S$ is $t$-type w.r.t. $F^2_{i',j'}$ for some $t\in [k]$; 
\item $S$ is  $i'$-type w.r.t. $F^1_{i',j'}$ and $j'$-type w.r.t. $F^2_{i',j'}$. 
\end{enumerate}

For convenience, we rearrange the indices in $M$ and write $M=[m]$. Let $\phi: V(F)\to [m]$ satisfying $\phi(v_\ell)=\ell$ for all $\ell\in [m]$.
 Given $S\in \partial F$, let $\phi (S)$ denote the $(k-1)$-set consisting of the subscripts of vertices in $S$. Now we consider $\psi:\partial F \to V(\mathcal A')$ as follows.
For each $S\in \partial F$, let $\psi(S)=\alpha^{r}_{\phi(S)}$ if $S$ is $r$-type w.r.t. $F^1_{i',j'}$ for some $r\in [k]\setminus \{i'\}$; let $\psi(S)=\alpha^{i'}_{\phi(S)}$ if $S$ is $i'$-type w.r.t. $F^1_{i',j'}$ or $j'$-type w.r.t. $F^2_{i',j'}$; let $\psi(S)=\beta^{t}_{\phi(S)}$ if 
 $S$ is $t$-type w.r.t. $F^2_{i',j'}$ for some $t\in [j'-1]$;
let $\psi(S)=\beta^{t-1}_{\phi(S)}$ if 
$S$ is $t$-type w.r.t. $F^2_{i',j'}$ for some $t\in [k]\setminus [j']$.
By the 
properties~\ref{p111} and~\ref{p222} in Lemma~\ref{lem-Embedding-reduced}, we obtain that $(\phi, \psi)$ is a reduced map from $F$ to $\mathcal A'$. Thus, $\mathcal A$ embeds $F$.
\end{proof}

To prove Lemma~\ref{lem-Embedding-reduced}, our main tool is the classical Ramsey theorem for multicolored
hypergraphs, which we state below for reference.

\begin{theorem}[Ramsey~\cite{Ramsey}]  \label{thm-Ramsey} 
	For any $r_R, k_R, n_R \in \mathbb N$, there exists $N \in \mathbb N$ such that every $r_R$-edge-coloring of a $k_R$-uniform clique with $N$ vertices contains a monochromatic $k_R$-uniform clique with $n_R$ vertices.
\end{theorem}


Next, we state and prove several lemmas that are useful for the proof of  
Lemma~\ref{lem-Embedding-reduced}.
For convenience, we start with some useful notation. 
Recalling that  each constituent of reduced $k$-graphs $\mathcal A$ is always a $k$-partite  $k$-graph.
For convenience,
we consider the {\it normalized degree} of each vertex of  $\mathcal A$ as follows.
Given an $m$-reduced $k$-graphs $\mathcal A$, a $k$-set ${\mathcal Y}=\llbracket y_1, y_2,\dots, y_k\rrbracket \in [m]^k$ and a coordinate $y_{\ell}\in \mathcal Y$ with $\ell\in [k]$, we define 
\[
\deg_{{\mathcal Y} \to y_{\ell} }(v):=\frac{|\{e\in E(\mathcal A_{\mathcal Y}): v\in e\}|} {\prod_{j \in [k]\setminus \{\ell\}}|\mathcal P_{ \mathcal Y \setminus \{y_{j}\}}|}
\]
for each  $v\in \mathcal P_{ \mathcal Y \setminus \{y_{\ell}\}}$.
Moreover, for $\rho>0$, let
\[
\mathcal S^{\rho}_{\mathcal Y\setminus \{y_\ell\}\to y_\ell}:=\{v\in \mathcal P_{ \mathcal Y \setminus \{y_{\ell}\}}: \deg_{\mathcal Y \to y_{\ell}}(v)\ge \rho \}.
\]
To make our notation easier to follow, we refer to vertices that belong to $\mathcal P_{ \mathcal Y \setminus \{y_{\ell}\}}$  as {\it $\ell$-type} (w.r.t. some $\mathcal Y=\llbracket y_1, y_2,\dots, y_k\rrbracket \subset \mathbb N$). 

The following lemmas explore that for sufficiently large $N\in \mathbb N$, each $N$-reduced $k$-graph $\mathcal A$ contains an induced subhypergraph $\mathcal A'$ such that for each $\ell\in [k]$, the proportions of $\ell$-type vertices with a non-negligible  normalized degree in all constituents of  $\mathcal A'$
are approximately the same.

\begin{lemma}\label{lem-1}
	Given $\rho>0$ and integers $ m^*\ge k \ge 3$, there exists $N\in \mathbb N$ such that the following holds.
	For every $N$-reduced $k$-graph  $\mathcal A$,  there exist constants $t_{\ell}$ for $\ell \in [k]$, and there exists an induced subhypergraph $\mathcal A'\subseteq \mathcal A$ on set $M^*\subseteq [N]$ with $|M^*|=m^*$ such that for every $k$-set ${\mathcal Y}=\llbracket y_1, y_2,\dots, y_k\rrbracket \in [M^*]^k$ the following holds
	\begin{equation}\label{eq:almost-same}
	t_{\ell} |\mathcal P_{ \mathcal Y \setminus \{y_{\ell}\}}|\le |\mathcal S^{\rho}_{\mathcal Y\setminus \{y_\ell\}\to y_\ell}| < (t_{\ell}+ \rho)  |\mathcal P_{ \mathcal Y \setminus \{y_{\ell}\}}| \text{~~for~every~} \ell\in [k].
	\end{equation}	
	
\end{lemma}

\begin{proof}
	We apply Theorem~\ref{thm-Ramsey} with $r_R= (\lfloor \rho^{-1} \rfloor +1)^k$, $k_R=k$ and $n_R= m^*$ to get $N\in \mathbb N$. Let $\mathcal A$ be an $N$-reduced $k$-graph. 
Let us consider  an $r_R$-edge-coloring $k$-uniform clique
	with vertex set $[N]$ as follows. 
	For every ${\mathcal Y}=\llbracket y_1, y_2,\dots, y_k\rrbracket \in [N]^k$ and $\ell\in [k]$,
	we color  ${\mathcal Y}$ with the
	triple 
	\[
	\left(\left\lfloor \frac{|\mathcal S^{\rho}_{\mathcal Y\setminus \{y_1\}\to y_1}|}{\rho |\mathcal P_{ \mathcal Y \setminus \{y_{1}\}}|} \right\rfloor, \left\lfloor \frac{|\mathcal S^{\rho}_{\mathcal Y\setminus \{y_2\}\to y_2}|}{\rho |\mathcal P_{ \mathcal Y \setminus \{y_{2}\}}|} \right\rfloor, \dots, \left\lfloor \frac{|\mathcal S^{\rho}_{\mathcal Y\setminus \{y_k\}\to y_k}|}{\rho |\mathcal P_{ \mathcal Y \setminus \{y_{k}\}}|} \right\rfloor \right).
	\]
By Theorem~\ref{thm-Ramsey}, there exists a subset
	$M^*\subseteq [N]$ with $|M^*|=m^*$ such that all $k$-sets induced on $M^*$ have the same color, say $(t'_1,  t'_2, \dots,  t'_k)$.
	Therefore, the induced subhypergraph $\mathcal A'$ on set $M^*$ satisfies the statement of the lemma with $t_{\ell}=\rho t'_{\ell}$ for $\ell \in [k]$.	
\end{proof}

Using Lemma~\ref{lem-1}, we shall show that every $N$-reduced $k$-graph $\mathcal A$ contains a well-behaved induced subhypergraph when its density larger than
$k^{-k}$.

\begin{lemma}\label{lem-2}
	Given  $\eps>0$, there exists  $\rho >0$ such that for integers $ m> k \ge 3$, there exists $N\in \mathbb N$ such that the following holds.
	For every $(k^{-k}+\eps)$-dense $N$-reduced $k$-graph $\mathcal A$,  there exists an induced subhypergraph $\mathcal A'\subseteq \mathcal A$ on set $M\subseteq [N]$ with $|M|=m$ that satisfies the following property:
	\begin{itemize}
		\item  There is a pair $\{i',j'\}\in [k]^2$ with $i'<j'$ such that for all $I=\llbracket z_1, z_2, \dots, z_{k+1}\rrbracket\in [M]^{k+1}$ with $\mathcal X:=I \setminus \{z_{i'}, z_{j'+1}\}$, we have 
\[
|\mathcal S^{\rho}_{\mathcal X\to z_{i'} }\cap \mathcal S^{\rho}_{\mathcal X\to z_{j'+1} }|\ge \rho |\mathcal P_{ \mathcal X}|.
\]
	\end{itemize} 		
\end{lemma}

\begin{proof}
	Given $\eps >0$ (without loss of generality let $\eps<1/2$),  let $\rho=\frac{\eps}{k^k}$ and  $\rho_0=\rho\binom{k}{2}$.
	We first apply Theorem~\ref{thm-Ramsey} with  $r_R= \binom{k}{2}$, $k_R=2k-1$ and $n_R = 2m+1$ to get $N'\in \mathbb N$.
 Then we apply Lemma~\ref{lem-1} with $\rho_0$ and $m^*=N'$ to get $N$.
	
	Let  $\mathcal A$ be a $(k^{-k}+\eps)$-dense $N$-reduced $k$-graph and let $\mathcal A^*$ be the induced subhypergraph on set $M^*$ with $|M^*|=m^*$ provided by Lemma~\ref{lem-1} along with the
	reals $t_{\ell}$ for $\ell\in [k]$ with the properties given in the statement of Lemma~\ref{lem-1}.
	Then, we first claim that 
\begin{equation}\label{eq:t-sum}
\sum_{\ell\in [k]} t_{\ell} \ge 1+\rho_0.
\end{equation}
	If not, suppose that $\sum_{\ell=1}^{k} t_{\ell} < 1+\rho_0$. Given ${\mathcal Y}=\llbracket y_1, y_2,\dots, y_k\rrbracket\in [M^*]^k$ and $\ell\in [k]$, we have $\deg_{\mathcal Y \to y_{\ell}}(u)< \rho_0$ for each vertex $u\in \mathcal P_{ \mathcal Y \setminus \{y_{\ell}\}}\setminus \mathcal S^{\rho_0}_{\mathcal Y\setminus \{y_\ell\}\to y_\ell}$. 
By Lemma~\ref{lem-1}, we obtain that
\begin{align*}
|E(\mathcal A_{\mathcal Y})|&< 
\sum_{\ell=1}^{k} \left|\mathcal P_{ \mathcal Y \setminus \{y_{\ell}\}}\setminus \mathcal S^{\rho_0}_{\mathcal Y\setminus \{y_\ell\}\to y_\ell}\right|\cdot\left(\rho_0 \prod_{j \in [k]\setminus \{\ell\}}|\mathcal P_{ \mathcal Y \setminus \{y_{j}\}}|\right)
+ \prod_{\ell \in [k]}\left|\mathcal S^{\rho_0}_{\mathcal Y\setminus \{y_\ell\}\to y_\ell}\right|\\
&\stackrel{\eqref{eq:almost-same}}{\le}
	\left(k\rho_0+ \prod_{\ell \in [k]}(t_{\ell}+\rho_0)\right) \prod_{j \in [k]}|\mathcal P_{ \mathcal Y \setminus \{y_{j}\}}|.
\end{align*}
By the AM-GM inequality, the edge density of $\mathcal A_{\mathcal Y}$  has no more than
	\[
	k\rho_0+\left(\frac{\sum_{\ell=1}^{k}(t_{\ell}+\rho_0)}{k}\right)^{k}\le k\rho_0+\left(\frac{1+(k+1)\rho_0}{k}\right)^{k}<k^{-k}+\eps,
	\]
	which contradicts that $\mathcal A$ is $(k^{-k}+\eps)$-dense, where the last inequality follows from  $\rho_0=\eps k^{1-k}$.

	We next consider  a $\binom{k}{2}$-edge-coloring  $(2k-1)$-uniform clique
on set $M^*$ as follows. 
	Given a $(2k-1)$-set  $Q=\llbracket y_1, x_1, y_2, x_2,\dots, y_{k-1}, x_{k-1}, y_k\rrbracket \subset M^*$ and $\mathcal X:=\llbracket x_1, x_2, \dots, x_{k-1}\rrbracket$, let the edge-coloring $\phi: [M^*]^{2k-1}\to [k]^2$ satisfy  the following Algorithm~\ref{algorithm 1}:
	\begin{algorithm}[ht]
		\caption{}
		\label{algorithm 1}
		Given a $(2k-1)$-set $Q$,	{\bf input: }  $\mathcal S^{\rho}_{\mathcal X \to \{y_\ell\}}$ for $\ell\in [k]$, and {\bf output:} $\phi(Q)$.
		\begin{algorithmic}[1]
		\For{$i = 1$ to $k -1$}
		\For{$j = i + 1$ to $k$}
		\If{$|\mathcal S^{\rho}_{\mathcal X \to y_i} \cap \mathcal S^{\rho}_{\mathcal X \to y_j}| \ge \rho|\mathcal P_{\mathcal X}|$}
		\State \Return $\{i, j\}$
		\EndIf
		\EndFor
		\EndFor
		\end{algorithmic}
	\end{algorithm}

	We claim that Algorithm~\ref{algorithm 1} is valid. If not,
	suppose that $|\mathcal S^{\rho}_{\mathcal X \to y_i} \cap \mathcal S^{\rho}_{\mathcal X \to y_j}|< \rho|\mathcal P_{\mathcal X}|$ for all $\{i,j\}\in [k]^2$, then we have
\begin{align*}
|\mathcal P_{\mathcal X}|  &\ge 
|\mathcal S^{\rho}_{\mathcal X \to y_1} \cup  \dots \cup \mathcal S^{\rho}_{\mathcal X \to y_k}|
\ge  \sum_{\ell\in [k]}|\mathcal S^{\rho}_{\mathcal X \to y_\ell}|- \sum_{\{i,j\}\in [k]^2} |\mathcal S^{\rho}_{\mathcal X \to y_i} \cap \mathcal S^{\rho}_{\mathcal X \to y_j}|\\
&\stackrel{(\rho_0>\rho)}{\ge} \sum_{\ell\in [k]}|\mathcal S^{\rho_0}_{\mathcal X \to y_\ell}|- \sum_{\{i,j\}\in [k]^2} |\mathcal S^{\rho}_{\mathcal X \to y_i} \cap \mathcal S^{\rho}_{\mathcal X \to y_j}|
\stackrel{\eqref{eq:almost-same}}{>}
\left(\sum_{\ell\in [k]}t_{\ell}-\rho_0\right)|\mathcal P_{\mathcal X}|\stackrel{\eqref{eq:t-sum}}{\ge} |\mathcal P_{\mathcal X}|,
\end{align*}
which is impossible. By Theorem~\ref{thm-Ramsey}, we would obtain a set
	$M_0\subseteq M^*$ with $|M_0|=2m+1$ such that all $(2k-1)$-sets induced on $M_0$ have the same color, say $\{i',j'\}$ with $i'<j'$.
	
	For convenience, set $M_0=[2m+1]$. We choose $M=\{2, 4, 6, \dots, 2m\}\subset M_0$.
	Let $\mathcal A'$ be the induced subhypergraph of $\mathcal A$ with index set $M$. Then $\mathcal A'$  satisfies the statement of the lemma. Indeed, for any $I=\llbracket z_1, {z_2}, \dots, {z_{k+1}} \rrbracket\in [M]^{k+1}$, 
	we can extend  $I$ to a $(2k-1)$-set $Q=\llbracket y_1, x_1, y_2, x_2,\dots, y_{k-1}, x_{k-1}, y_k \rrbracket$ by adding elements from $M_0$ such that   $\{ x_1,x_2, \dots, x_{k-1} \}=I\setminus \{z_{i'}, z_{j'+1}\}$ and $y_{i'}=z_{i'}$ and $y_{j'}=z_{j'+1}$.
	Set $\mathcal X= I \setminus \{z_{i'}, z_{j'+1}\}$. Since $\phi(Q)=\{i',j'\}$,  we have 
\[
|\mathcal S^{\rho}_{\mathcal X\to z_{i'} }\cap \mathcal S^{\rho}_{\mathcal X\to z_{j'+1} }|\ge \rho |\mathcal P_{ \mathcal X}|.
\]
\end{proof}

Recalling the statement of Lemma~\ref{lem-Embedding-reduced}, we need to find a ``well-behaved" induced subhypergraph $\mathcal A' $ in $(k^{-k}+\eps)$-dense $N$-reduced $k$-graph $\mathcal A $, where ``well-behaved" means each vertex class $\mathcal P_{\mathcal X}$ of $\mathcal A' $ can find $(2k-1)$ vertices satisfying 
properties~\ref{p111} and~\ref{p222} in Lemma~\ref{lem-Embedding-reduced}. 
Given a reduced $k$-graph  $\mathcal A$, with candidate sets of good properties for each constituent, such as consider  $\ell$-type vertices with non-negligible normalized degree for each constituent, we aim to choose a representative vertex from each vertex class of $\mathcal A$ that possesses good properties for each constituent it belongs to. Since each vertex class of $\mathcal A$ may belong to many different constituents, it is possible that no single vertex is suitable for all constituents of $\mathcal A$, even if the size of the candidate sets relative to each constituent involving the vertex class is linearly proportional. However, by leveraging the power of Ramsey theory, it is possible to find such a representative vertex by passing to the induced subhypergraph of $\mathcal A$.
The following lemma is intended to identify such representative vertices.

\begin{lemma}\label{lem-3}
	Given  $\rho>0$, integers $ m\ge k \ge 3$ and  $t\in [k]$, there exists  $N\in \mathbb N$ such that the following holds.
	If $\mathcal A$ is an $N$-reduced $k$-graph and for each $k$-set $\mathcal Y=\llbracket y_1, y_2, \dots, y_k \rrbracket \in [N]^k$, then there is a subset $\mathcal S_{\mathcal Y\setminus \{y_t\}\to y_t}\subseteq  \mathcal P_{\mathcal Y\setminus \{y_t\}}$ satisfying 
	$|\mathcal S_{\mathcal Y\setminus \{y_t\}\to y_t}|\ge \rho |\mathcal P_{\mathcal Y\setminus \{y_t\}}|$.
	Then there exists an induced subhypergraph $\mathcal A'\subseteq \mathcal A$ on index set $M\subseteq [N]$ with $|M|=m$ and vertices $v_{\mathcal X}\in \mathcal P_{\mathcal X}$ for all $\mathcal X\in [M]^{k-1}$ such that  the following property holds:
	\begin{enumerate}	
		\item[$\bullet$]  For each $\mathcal X=\llbracket x_1, \dots, x_{k-1}\rrbracket \in [M]^{k-1}$, the vertex $v_{\mathcal X}$ satisfies 
		\[
		v_{\mathcal X}\in \bigcap_{
				x_{t-1}<y<x_{t},~
				y\in M} \mathcal S_{\mathcal X \to y},
		\]
		where $x_{t-1}= 0$ for $t=1$ and $x_t= N$ for $t=k$. 
	\end{enumerate}	
\end{lemma}

\begin{proof}
	Given $\rho >0$, $ m\ge k \ge 3$ and $t\in [k]$, we apply Theorem~\ref{thm-Ramsey} with $r_R= 2$  $k_R=m$ and $n_R =\max \{m^2, 2 \lceil\frac{m}{\rho} \rceil\} $ to get $N\in \mathbb N$. 
	Let $\mathcal A$ be an $N$-reduced $k$-graph  satisfying the condition of the lemma.
	We now construct a  $2$-edge-coloring  $m$-uniform clique on the vertex set $[N]$ as follows.
	For each $m$-set $Q=\llbracket a_1, a_2, \dots, a_m \rrbracket\in  [N]^{m}$, 
	let $L=\llbracket a_t, a_{t+1}, \dots, a_{t+m-k}\rrbracket$ and $\mathcal X= Q\setminus L$. Clearly, $\mathcal X\in [N]^{k-1}$. We say that 
	$Q$ is colored blue if there exists a vertex $v_{\mathcal X} \in \mathcal P_{\mathcal X}$ such that 
	$v_{\mathcal X} \in \bigcap_{\ell \in L} \mathcal S_{\mathcal X\to \ell}$;
	otherwise, $Q$ is colored red.

	By Theorem~\ref{thm-Ramsey}, there exists a set $S\subseteq [N]$ with $|S|=n_R$ such that all edges $Q$ induced on set $S$ have same color. For convenience, we rearrange the indices in $S$ and write $S=[n_R]$.
Let $J=\llbracket t, t+1, \dots, {n_R-k+t}\rrbracket$ and $\mathcal X=[n_R]\setminus J$.
By the condition of lemma, each set
$\mathcal S_{\mathcal X \to j}\subseteq  \mathcal P_{\mathcal X}$ for $j\in J$ satisfies $|\mathcal S_{\mathcal X \to j}|\ge \rho |\mathcal P_{\mathcal X}|$.
Since $n_R\ge 2\lceil\frac{m}{\rho} \rceil$,  by double counting, there exists a vertex $v_{\mathcal X} \in \mathcal P_{\mathcal X}$ and a subset $I\subseteq J $ with $|I|=  m-k+1$ such that 
\[
v_{\mathcal X} \in \bigcap_{i\in I} \mathcal S_{\mathcal X \to i},
\]
which implies that the common color for $m$-sets induced on set $S$ is  blue. 

Now, we choose $M=\{m, 2m \dots, m^2\}$. Clearly, $M\subset S$.
	For each $\mathcal X=\llbracket x_1, x_2, \dots, x_{k-1} \rrbracket \in  [M]^{k-1}$,
	we  extend  $\mathcal X$ to an $m$-set $Q$ by adding $(m-k+1)$-elements from $\{x_{t-1}+1, x_{t-1}+2, \dots, x_{t}-1\}$  such that $Q$ contains all elements in set $\{x\in M: x_{t-1}<x<x_t\}$. Since $Q$ is colored blue, 
	there is a vertex $v_{\mathcal X}\in \mathcal P_{\mathcal X}$  satisfying
	\[
	v_{\mathcal X}\in \bigcap_{
			x_{t-1} <y<x_t,~y\in M } \mathcal S_{\mathcal X \to y}.	
	\]
\end{proof}

The following lemma is designed to select the representative vertices based on sets $
\mathcal S^{\rho}_{\mathcal X\to z_{i'} }\cap \mathcal S^{\rho}_{\mathcal X\to z_{j'+1}}$ given by Lemma~\ref{lem-2}.

\begin{lemma}\label{lem-4}
Given $\rho>0$, integers $ m>k \ge 3$ and $\{i',j'\}\in [k]^2$ with $i'<j'$, there exists  $N\in \mathbb N$ such that the following holds.
Let $\mathcal A$ is an $N$-reduced $k$-graph. If for each $I=\llbracket z_1, z_2, \dots, z_{k+1} \rrbracket\in  [N]^{k+1}$ and $\mathcal X =I \setminus \{z_{i'}, z_{j'+1}\}$, there are subsets $\mathcal S_{z_{i'}\leftarrow \mathcal X\rightarrow z_{j'+1}}\subseteq \mathcal P_{\mathcal X}$ satisfying 
	$|\mathcal S_{z_{i'}\leftarrow \mathcal X\rightarrow z_{j'+1}} |\ge \rho |\mathcal P_{\mathcal X}|$.
	Then there exists an induced subhypergraph $\mathcal A'\subseteq \mathcal A$ on set $M\subseteq [N]$ with $|M|=m$ and vertices $v_{\mathcal X}\in \mathcal P_{\mathcal X}$ for all $\mathcal X\in [M]^{k-1}$ such that  the following property holds:
\begin{enumerate}	
\item[$\bullet$]  For each $\mathcal X=\llbracket x_1, \dots, x_{k-1}\rrbracket \in  [M]^{k-1}$, the vertex $v_{\mathcal X}$ satisfies 
\[
v_{\mathcal X}\in \bigcap_{
	\begin{array}{c}
		x_{i'-1}<y<x_{i'},~x_{j'-1}<y'<x_{j'} \\
		\{y,y'\}\subset M
\end{array} } \mathcal S_{y \leftarrow \mathcal X\rightarrow y'},
\]
		where $x_{i'-1}= 0$ for $i'=1$ and $x_{j'}= N$ for $j'=k$.
	\end{enumerate}	
\end{lemma}

\begin{proof}
	Given $\rho >0$, $ m\ge k \ge 3$ and $\{i',j'\}\in [k]^2$ with $i'<j'$, we apply Theorem~\ref{thm-Ramsey} with $r_R= 2$,  $k_R=2m-k+1$ and $n_R =\max\{m^2, 3\lceil\frac{m}{\rho} \rceil\}  $ to get $N\in \mathbb N$. 
	Let $\mathcal A$ be an $N$-reduced $k$-graph satisfying the condition of the lemma.
	We now consider a $2$-edge-coloring $k_R$-uniform clique on the vertex set $[N]$ as follows.
	For any $(2m-k+1)$-set $Q=\llbracket a_1, a_2, \dots, a_{2m-k+1} \rrbracket\subset [N]$, let $L_1=\llbracket a_{i'}, a_{i'+1}, \dots, a_{i'+m-k}\rrbracket$, $L_2=\llbracket a_{j'+m-k+1},  a_{j'+m-k+2}, \dots,  a_{j'+2m-2k+1} \rrbracket$ and $\mathcal X= Q\setminus (L_1\cup L_2)$.
	We define that $Q$ is colored blue if there exists a vertex $v_{\mathcal X} \in \mathcal P_{\mathcal X}$ such that 
	\[
	v_{\mathcal X}\in \bigcap_{\ell\in L_1,~\ell'\in L_2} \mathcal S_{\ell \leftarrow \mathcal X\rightarrow \ell'},
	\]
	otherwise, $Q$ is colored red.

	By Theorem~\ref{thm-Ramsey}, there exists a set $S\subseteq [N]$ with $|S|=n_R$ such that all edges $Q$ induced on set $S$ have same color. For convenience, we rearrange the indices in $S$ and write $S=[n_R]$.
 Let $J_1=\llbracket i', i'+1,  \dots, i'+ \lceil\frac{m}{\rho} \rceil-1 \rrbracket$,  $J_2=\llbracket j'+ \lceil\frac{m}{\rho} \rceil, j'+ \lceil\frac{m}{\rho} \rceil+1, \dots, j'+ 2\lceil\frac{m}{\rho} \rceil-1 \rrbracket$ and $\mathcal X=[2\lceil\frac{m}{\rho} \rceil+k-1]\setminus (J_1\cup J_2)$.
Since $n_R =\max\{m^2, 3\lceil\frac{m}{\rho} \rceil\}$, we have $J_1, J_2, \mathcal X \subset S$.
By the condition of lemma,  we have
\[
|\mathcal S_{i'-1+j \leftarrow \mathcal X\rightarrow j'-1+\lceil \frac{m}{\rho} \rceil+j}|\ge \rho |\mathcal P_{\mathcal X}|, \text{~for~} j\in [\lceil \frac{m}{\rho} \rceil].
\]
 An easy averaging argument shows that there exists a vertex $v_{\mathcal X} \in \mathcal P_{\mathcal X}$ and a subset $J\subset [\lceil \frac{m}{\rho} \rceil]$ with $|J|= m-k+1$ such that 
\[
v_{\mathcal X} \in \bigcap_{j\in J}\mathcal S_{i'-1+j \leftarrow \mathcal X\rightarrow j'-1+\lceil \frac{m}{\rho} \rceil+j},
\]
which implies that the common color for the $(2m-k+1)$-sets induced on set $S$ is  blue.

 Now we choose $M=\{m, 2m \dots, m^2\}$. 
	For each $\mathcal X=\llbracket x_1, x_2, \dots, x_{k-1} \rrbracket \in  [M]^{k-1}$,
	we extend  $\mathcal X$ to a $(2m-k+1)$-set $Q$ by adding elements from sets $\{x \in S: x_{i'-1}<x<x_{i'}\}$ and  $\{x'\in S: x_{j'-1}<x'<x_{j'}\}$
	such that $Q$ contains all elements in $\{y\in M: x_{i'-1}<y<x_{i'} \text{~or~} x_{j'-1}<y<x_{j'} \}$. Since $Q$ is colored blue,  
	there is a vertex $v_{\mathcal X}$  satisfying
\[
v_{\mathcal X}\in \bigcap_{
	\begin{array}{c}
	x_{i'-1}<y<x_{i'},~x_{j'-1}<y'<x_{j'} \\
		\{y,y'\}\subset M
\end{array} } \mathcal S_{y \leftarrow \mathcal X\rightarrow y'}.
\]
\end{proof}

To prove Lemma~\ref{lem-Embedding-reduced}, we also need an auxiliary lemma for $k$-partite  $k$-graphs.

\begin{lemma}\label{lem-5}
	For any $\rho>0$, $k\ge 3$ and $t\in [k-1]$, the following holds for every $k$-partite  $k$-graph $H$  with vertex partition $\{V_1, \dots, V_k\}$.
	Let $T=\{v_1, \dots, v_t\}$  be a subset of $V(H)$ with $v_i\in V_i$ for $i\in [t]$.
	If $T$ is contained in at least $\rho \prod_{j\in [k]\setminus [t]}|V_j|$ 
	edges of $H$, then there exist at least $\frac{\rho}{2} |V_{t+1}|$ vertices $u\in V_{t+1}$  such that $T \cup \{u\}$ is
	contained together in at least $\frac{\rho}{2} \prod_{j\in [k]\setminus [t+1]}|V_j|$  edges of $H$.
\end{lemma}

\begin{proof}
Given $T=\{v_1, \dots, v_t\}$ with $v_i\in V_i$ for $i\in [t]$, let 
\[
U_{t+1}=\{ u\in V_{t+1}:  T \cup \{u\} \text{~is~contained~in~at~least~} \frac{\rho}{2} \prod_{j\in [k]\setminus [t+1]}|V_j| \text{~edges~of~} H  \}.
\]  
 If $|U_{t+1}|<\frac{\rho}{2} |V_{t+1}|$, then the number of edges in $H$ containing $T$ is at most
	\begin{align*}
 |U_{t+1}|\cdot \prod_{j\in [k]\setminus [t+1]}|V_j|+|V_{t+1}\setminus U_{t+1} |\cdot \frac{\rho}{2} \prod_{j\in [k]\setminus [t+1]}|V_j| < \rho  \prod_{j\in [k]\setminus [t]}|V_j|,
	\end{align*}
	which contradicts the assumption of the lemma.	
\end{proof}

Using Lemmas~\ref{lem-2} --~\ref{lem-5}, we are now ready to prove Lemma~\ref{lem-Embedding-reduced}.

\begin{proof}[Proof of Lemma~\ref{lem-Embedding-reduced}]

We begin by outlining the main
ideas of this proof. The argument proceeds in $2k$ stages.
Given $\eps>0$ and $m>k\ge 3$, we choose constants satisfying the following hierarchy (form right to left):
\[
N^{-1}\ll m^{-1}_{2k-1}\ll m^{-1}_{2k-2} \ll \dots \ll m^{-1}_{1} \ll m^{-1}, \eps <1/2,
\]
and $\rho=\eps k^{-k}$.
Let $\mathcal A$ be a $(k^{-k}+\eps)$-dense $N$-reduced $k$-graph.
	In the $1$\textsuperscript{th} step, by  Lemma~\ref{lem-2}, we can get an index  set $M_{2k-1}\subseteq [N]$ of size $m_{2k-1}$   and a pair $\{i',j'\}\in [k]^2$ with $i'<j'$ such that the induced subhypergraph $\mathcal A_{2k-1}\subseteq \mathcal A$  on  $M_{2k-1}$ satisfies conditions in Lemma~\ref{lem-4}. 
	Then, in the next stage, using Lemma~\ref{lem-4}, we can shrink the index set $M_{2k-1}$ to some $M_{2k-2}\subseteq M_{2k-1}$ of size $m_{2k-2}$ and get vertices $\alpha^{i'}_{\mathcal X}\in \mathcal P_{\mathcal X}$ for all $\mathcal X\in [M_{2k-2}]^{k-1}$ such that the induced subhypergraph $\mathcal A^{2k-2}\subseteq \mathcal A^{2k-1}$ on  $M_{2k-2}$ and vertices $\alpha^{i'}_{\mathcal X}$ with $\mathcal X\in [M_1]^{k-1}$ satisfy the property in Lemma~\ref{lem-4}.
Next consider the vertices $\alpha^{i'}_{\mathcal X}$ as $i'$-type vertices. Since each $\alpha^{i'}_{\mathcal X}$ has a non-negligible  normalized degree,  by Lemma~\ref{lem-5},
we always choose  linearly proportional vertices such that their combined normalized degree is non-negligible. 
We can then use Lemma~\ref{lem-3} to shrink such candidate sets to obtain vertices $\alpha^{\ell}_{\mathcal X}$ for some $\ell\in [k]\setminus \{i'\}$. This iterative process (using Lemma~\ref{lem-5} and Lemma~\ref{lem-3} alternately) can continue for $(k-1)$ steps until we have selected all desired vertices $\alpha^1_{\mathcal X}, \dots, \alpha^{i'-1}_{\mathcal X}, \alpha^{i'+1}_{\mathcal X},\dots, \alpha^k_{\mathcal X}$.
We then consider the vertices $\alpha^{i'}_{\mathcal X}$ as $j'$-type vertices and similarly obtain the desired vertices $\beta^1_{\mathcal X}, \dots, \beta^{k-1}_{\mathcal X}$ separately in each subsequent step.

	Let $\mathcal A$ be a $(k^{-k}+\eps)$-dense $N$-reduced $k$-graph. In the beginning, we apply Lemma~\ref{lem-2} with $\eps$ to get a constant $\rho=\eps k^{-k}$. 
	By  Lemma~\ref{lem-2},  there exists an induced subhypergraph $\mathcal A_{2k-1}\subseteq \mathcal A$  on set $M_{2k-1}\subseteq [N]$ of size $m_{2k-1}$  and a pair $\{i',j'\}\in [k]^2$ with $i'<j'$ such that the following property holds: 
	\begin{itemize}
	\item For each $I=\llbracket z_1, z_2, \dots, z_{k+1}\rrbracket\in [M_{2k-1}]^{k+1}$ with $\mathcal X:=I \setminus \{z_{i'}, z_{j'+1}\}$, we have 
	\[
	|\mathcal S^{\rho}_{\mathcal X\to z_{i'} }\cap \mathcal S^{\rho}_{\mathcal X\to z_{j'+1} }|\ge \rho |\mathcal P_{ \mathcal X}|.
	\]
\end{itemize} 		
In the $2$\textsuperscript{th} step, we apply Lemma~\ref{lem-4} with $\rho$ to $\mathcal A_{2k-1}$ and sets  $\mathcal S^{\rho}_{\mathcal X\to z_{i'} }\cap \mathcal S^{\rho}_{\mathcal X\to z_{j'+1} }$ to get an induced subhypergraph $\mathcal A_{2k-2}\subseteq \mathcal A_{2k-1}$  on set $M_{2k-2}\subseteq M_{2k-1}$ of size $m_{2k-2}$, and vertices $\alpha^{i'}_{\mathcal X}\in \mathcal P_{\mathcal X}$ for all $ \mathcal X=\llbracket x_1, \dots, x_{k-1}\rrbracket \in [M_{2k-2}]^{k-1}$ such that  
		\[
		\alpha^{i'}_{\mathcal X}\in \bigcap_{
			\begin{array}{c}
				x_{i'-1}<y<x_{i'},~x_{j'-1}<y'<x_{j'} \\
				\{y,y'\}\subset M_{2k-2}
		\end{array} } \mathcal S^{\rho}_{\mathcal X\to y }\cap \mathcal S^{\rho}_{\mathcal X\to y' } ,
		\]
where $x_{i'-1}= 0$ for $i'=1$ and $x_{j'}= N$ for $j'=k$.

For convenience, we may assume without loss of generality that $i'=1$ and $j'=2$.
Until now, we have vertices $\alpha^{1}_{\mathcal X}\in \mathcal P_{\mathcal X}$ for all $ \mathcal X=\llbracket x_1, \dots, x_{k-1}\rrbracket \in [M_{2k-2}]^{k-1}$. Moreover, for any $0<y<x_{2}$ and $x_{1}<y'<x_{2}$ with $\{y,y'\}\subset M_{2k-2}$, we have 
\begin{equation}\label{equ-pf-lem-41-1}
	\deg_{{\mathcal X}\cup \{y\} \to y}(\alpha^{1}_{\mathcal X})\ge \rho \text{~~and~~}
	\deg_{{\mathcal X}\cup \{y'\} \to y'}(\alpha^{1}_{\mathcal X})\ge \rho.
\end{equation}

Next, for each $\mathcal Y=\llbracket y_1, y_2, \dots, y_k \rrbracket \in [M_{2k-2}]^k$ and $\mathcal X_\ell=\mathcal Y\setminus \{y_\ell\}$,  let 
\[
U_{\mathcal X_2\to y_2}=\{ u\in \mathcal P_{\mathcal X_2}:  \{\alpha^{1}_{\mathcal X_1}, u\} \text{~is~contained~in~at~least~} \frac{\rho}{2} \prod_{j\in [k]\setminus \{1,2\}}|\mathcal P_{\mathcal X_j}| \text{~edges~of~} \mathcal A_{\mathcal Y} \}.
\]  
Due to $\deg_{{\mathcal Y} \to y_{1}}(\alpha^{1}_{\mathcal X_1})\ge \rho$  (see~\eqref{equ-pf-lem-41-1}), we have   $|U_{\mathcal X_2\to y_2}|\ge \frac{\rho}{2}|\mathcal P_{\mathcal X_2}|$ by Lemma~\ref{lem-5}.
In the  $3$\textsuperscript{th} step,
we apply  Lemma~\ref{lem-3} with $\rho/2$ to $\mathcal A_{2k-2}$ and sets  $U_{\mathcal X_2\to y_2}$ to get  an induced subhypergraph $\mathcal A_{2k-3}\subseteq \mathcal A_{2k-2}$  on set $M_{2k-3}\subseteq M_{2k-2}$ of size $m_{2k-3}$ and 
	vertices $\alpha^{2}_{\mathcal X}\in \mathcal P_{\mathcal X}$ for all $\mathcal X\in [M_{2k-3}]^{k-1}$ such that  for each $\mathcal X=\llbracket x_1, \dots, x_{k-1}\rrbracket \in [M_{2k-3}]^{k-1}$, the vertex $\alpha^{2}_{\mathcal X}$ satisfies 
\[
\alpha^{2}_{\mathcal X} \in 
\bigcap_{x_1<y<x_2,~ y\in M_{2k-3} } U_{\mathcal X\to y}.
\]

Now suppose that the step $t$ for $3\le t\le k$  has been finished. We have obtained  the induced subhypergraph $\mathcal A_{2k-t}$ on set $M_{2k-t}$ of size $m_{2k-t}$, as well as vertices  $\alpha^{1}_{\mathcal X}, \alpha^{2}_{\mathcal X}, \dots,  \alpha^{t-1}_{\mathcal X} \in \mathcal P_{\mathcal X}$ for all $\mathcal X\in [M_{2k-t}]^{k-1}$. In particular, for each $\mathcal Y=\llbracket y_1, y_2, \dots, y_k \rrbracket \in [M_{2k-t}]^k$ with $\mathcal X_\ell=\mathcal Y\setminus \{y_\ell\}$, we have 
\[
|\{e\in E(\mathcal A_{\mathcal Y}): \{\alpha^{1}_{\mathcal X_1}, \alpha^{2}_{\mathcal X_2}, \dots,  \alpha^{t-1}_{\mathcal X_{t-1}}\} \subset e\}|\ge \frac{\rho}{2^{t-2}} \prod_{j\in [k]\setminus  [t-1]} |\mathcal P_{\mathcal X_j}|.
\]
Next, 
for each $\mathcal Y=\llbracket y_1, y_2, \dots, y_k \rrbracket \in [M_{2k-t}]^k$ with $\mathcal X_\ell=\mathcal Y\setminus \{y_\ell\}$,  let 
\[
U_{\mathcal X_t\to y_t}=\{ u\in \mathcal P_{\mathcal X_t}:  \{\alpha^{1}_{\mathcal X_1},  \dots, \alpha^{t-1}_{\mathcal X_{t-1}}, u\} \text{~is~contained~in~at~least~} \frac{\rho}{2^{t-1}} \prod_{j\in [k]\setminus  [t]} |\mathcal P_{\mathcal X_j}| \text{~edges~of~} \mathcal A_{\mathcal Y} \}.
\]  
By Lemma~\ref{lem-5}, we have   $|U_{\mathcal X_t\to y_t}|\ge \frac{\rho}{2^t}|\mathcal P_{\mathcal X_t}|$.
	At the $(t+1)$\textsuperscript{th} step,
	we apply  Lemma~\ref{lem-3} with $\rho/2^t$ to $\mathcal A_{2k-t}$ and sets  $U_{\mathcal X_t\to y_t}$ to get  an induced subhypergraph $\mathcal A_{2k-t-1}\subseteq \mathcal A_{2k-t}$  on set $M_{2k-t-1}\subseteq M_{2k-t}$ of size $m_{2k-t-1}$ and 
	vertices $\alpha^{t}_{\mathcal X}\in \mathcal P_{\mathcal X}$ for all $\mathcal X\in [M_{2k-t-1}]^{k-1}$ such that  for each $\mathcal X=\llbracket x_1, \dots, x_{k-1}\rrbracket \in [M_{2k-t-1}]^{k-1}$, the vertex $\alpha^{t}_{\mathcal X}$ satisfies 
	\[
	\alpha^{t}_{\mathcal X} \in 
	\bigcap_{x_{t-1}<y<x_t,~ y\in M_{2k-t-1} } U_{\mathcal X\to y}.
	\]
Therefore, when the $(k+1)$\textsuperscript{th} step is over, we can obtain an induced subhypergraph $\mathcal A_{k-1}$ on set $M_{k-1}$ of size $m_{k-1}$ and vertices  $\alpha^{1}_{\mathcal X}, \alpha^{2}_{\mathcal X}, \dots,  \alpha^{k}_{\mathcal X}\in \mathcal P_{\mathcal X}$ for all $\mathcal X\in [M_{k-1}]^{k-1}$. 
In particular, for each $\mathcal Y\in [M_{k-1}]^k$ and $\mathcal X_{\ell}\in  [\mathcal Y]^{k-1}$ with $\ell \in [k]$, we have  $\{\alpha^1_{\mathcal X_{1}}, \alpha^2_{\mathcal X_{2}}, \dots, \alpha^k_{\mathcal X_{k}} \} \in E(\mathcal A_{\mathcal Y})$.

Next, for every $k$-set $\mathcal Y=\llbracket y_1, \dots, y_{k}\rrbracket \in  [M_{k-1}]^k$ and  $\mathcal X_{\ell}=\mathcal Y\setminus \{y_{\ell}\}$ with $\ell\in [k]$, let
\[
U'_{\mathcal X_1\to y_1}=\{ u'\in \mathcal P_{\mathcal X_1}:  \{u', \alpha^{1}_{\mathcal X_2}\} \text{~is~contained~in~at~least~} \frac{\rho}{2} \prod_{j\in [k]\setminus \{1,2\}}|\mathcal P_{\mathcal X_j}| \text{~edges~of~} \mathcal A_{\mathcal Y} \}.
\]
Recalling the conclusion (\ref{equ-pf-lem-41-1}), we  have $\deg_{{\mathcal Y} \to y_{2}}(\alpha^{1}_{\mathcal X_{2}})\ge \rho$. Thus, we have $| U'_{\mathcal X_1\to y_1}|\ge \frac{\rho}{2}|\mathcal P_{\mathcal X_{1}}|$ by Lemma~\ref{lem-5}.
By Lemma~\ref{lem-3} applied with $\mathcal A_{k-1}$ and sets $U'_{\mathcal X_1\to y_1}$, 
there exists  an induced  subhypergraph $\mathcal A_{k-2}\subseteq \mathcal A_{k-1}$ on set $M_{k-2}\subseteq M_{k-1}$ of size $m_{k-2}$, and there exist
vertices $\beta^{1}_{\mathcal X}\in \mathcal P_{\mathcal X}$ for all $\mathcal X\in [M_{k-2}]^{k-1}$ such that for each $\mathcal X=\llbracket x_1, \dots, x_{k-1}\rrbracket \in [M_{k-2}]^{k-1}$, the vertex $\beta^{1}_{\mathcal X}$ satisfies 
	\[
	\beta^{1}_{\mathcal X}\in \bigcap_{0<y<x_1,~ y\in M_{k-2} } U'_{\mathcal X\to y},
	\]
which means that for any $\mathcal Y=\llbracket y_1, \dots, y_{k}\rrbracket \in  [M_{k-2}]^k$ and  $\mathcal X_{\ell}=\mathcal Y\setminus \{y_{\ell}\}$ with $\ell\in [k]$, the pair $\{\beta^1_{\mathcal X_1}, \alpha^{1}_{\mathcal X_{2}} \}$
is contained in at least $\frac{\rho}{2}\prod_{i\in [k]\setminus \{1, 2\}}|\mathcal P_{\mathcal X_i}|$ edges of $\mathcal A_{\mathcal Y}$. 
Therefore, in each subsequent step $t$ with $k+3\le t\le 2k$, we consider the sets
\[
U'_{\mathcal X_{t-k}\to y_{t-k}}=\left\{ u'\in \mathcal P_{\mathcal X_t}: \big|\{e\in \mathcal A_{\mathcal Y}: \{\beta^{1}_{\mathcal X_1}, \alpha^{1}_{\mathcal X_2},\dots, \beta^{t-k-2}_{\mathcal X_{t-k-1}},u'\}\subseteq e\}\big|\ge \frac{\rho}{2^{t-k-1}} \prod_{j\in [k]\setminus [t-k]}|\mathcal P_{\mathcal X_j}|\right\}.
\] 
Similar to the process for choosing vertices $\alpha^{3}_{\mathcal X}, \dots, \alpha^{k}_{\mathcal X}$, we apply Lemma~\ref{lem-5} and Lemma~\ref{lem-3} with the sets $U'_{\mathcal X_{t-k}\to y_{t-k}}$ to get the induced subhypergraph $\mathcal A_{2k-t}\subseteq \mathcal A_{2k-t+1}$ on set $M_{2k-t}$ of size $m_{2k-t}$ and vertices $\beta^{t-k-1}_{\mathcal X}\in \mathcal P_{\mathcal X}$ for all $\mathcal X\in [M_{2k-t}]^{k-1}$.

After performing the procedure $2k$ steps as described, we obtain an induced subhypergraph $\mathcal A'$ (i.e., $\mathcal A'\subseteq \mathcal A_{1}\subseteq \dots \subseteq \mathcal A_{2k-1}\subseteq \mathcal A$) on set $M$ of size $m$ and vertices  $\alpha^{1}_{\mathcal X},\dots, \alpha^{k}_{\mathcal X}, \beta^{1}_{\mathcal X}, \dots,  \beta^{k-1}_{\mathcal X}\in \mathcal P_{\mathcal X}$ for all $\mathcal X\in [M]^{k-1}$, which
satisfy the properties~\ref{p111} and~\ref{p222} in this lemma.
\end{proof}

\section{Proof of Theorem~\ref{example-thm}}\label{sec-example}
In this section, we will prove Theorem~\ref{example-thm}.
Given a $k$-graph $F$, a necessary condition to prove that $\pi_{k-2}(F)\ge k^{k}$ is to show that $F$ has no vanishing ordering of $V(F)$. Since  there are $|V(F)|!$ ways to order $V(F)$, it is troublesome to check the ordering of $V(F)$ one by one according to the definition of ``vanishing ordering".
Therefore,  we will prove a lemma (see Lemma~\ref{example-lem}), which will be useful to rule out the existence of a vanishing ordering of
vertices of a $k$-graph. In particular,  Lemma~\ref{example-lem} is equivalent to~\cite[Lemma 16]{1/27} when $k=3$. 
As a  less obvious generalization of~\cite[Lemma 16]{1/27}, we start with introducing some notation.

Given $k\ge 2$, a {\it tight $k$-uniform cycle} $C_\ell^{(k)}$ of length $\ell> k$ is a sequence $(v_0, v_1,\dots, v_{\ell-1})$ of vertices, satisfying that
$\{v_i, \dots, v_{i+k-1}\}$ is an edge for every $0\le i\le \ell-1 $ with addition of indices taken modulo $\ell$.
A {\it $k$-uniform directed hypergraph} $D$ ({\it $k$-digraph} for short) is a pair  $D=(V(D), A(D))$ where $V (D)$ is a vertex set and $A(D)$ is a set of $k$-tuples
of vertices, called {\it  directed edge} set.
A {\it directed tight $k$-uniform cycle} $\vec C_\ell^{(k)}$  of length $\ell> k$ is  a sequence $(v_0, v_1,\dots, v_{\ell-1})$ of vertices, satisfying that
$(v_i, \dots, v_{i+k-1})$ is a directed edge for $0\le i\le \ell-1 $ (with addition of indices taken modulo $\ell$).
As usual 2-digraphs and directed tight 2-uniform cycles are simply called digraphs and directed cycles, respectively.
Given a $k$-digraph $D$, we define the 
{\it transitive digraph} $T(D)$ of $D$ as follows: $T(D)$ has the same vertex set as $D$, and each directed edge of
$D$ corresponds to a transitive tournament in $T(D)$, i.e., if $(x_1, x_2, \dots, x_k)\in A(D)$ then $(x_i,x_j)\in A(T(D))$ for any $1\le i<j\le k$.
In particular, a $k$-digraph $D$ is {\it simple} if at most one order of $k$-sets of its vertices is in $A(D)$.

\begin{lemma}\label{example-lem}
For $k\ge 3$, a $k$-graph $F$ has a vanishing ordering of $V(F)$ if and only if there exists
a $k$-edge-coloring simple $(k-1)$-digraph $D$ on $V(F)$ such that each $k$-edge of $F$ corresponds to a $k$-edge-coloring $\vec C_k^{(k-1)}$  with edges colored $0, 1,\dots  k-1$ (in
this order\footnote{This means that if  $k$-edge $e=\llbracket {v_1,v_2, \dots, v_k}\rrbracket$ under an ordering of $V(H)$, then $\vec C_k^{(k-1)}=(v_1,v_2, \dots, v_k) $ and
 the directed edge $(v_i, \dots, v_{i+k-2})$ is colored $i-1$ with addition of indices taken modulo $k$.}), and there exist two consecutive integers $\{\beta, \beta+1\}\subset \mathbb Z_{k}$ such that the subdigraph $D_{\beta, \beta+1}$ of $D$ containing all directed edges colored with $\beta$ or  $\beta+1$ satisfies the  following property:
\begin{enumerate}
\item[$\bullet$]  The transitive digraph $T(D_{\beta, \beta+1})$ does not contain directed cycles.
\end{enumerate}
\end{lemma}

In the proof of the Lemma~\ref{example-lem}, we will use a fundamental property of acyclic digraphs. Given a digraph $D$ and an ordering  $(v_1, v_2, \dots, v_n)$ of its vertices, we say  this ordering is {\it acyclic} if  for every directed edge $(v_i, v_j)\in A(D)$, we have $i<j$.

\begin{proposition}[{\cite[Proposition 2.1.3]{bang2008digraphs}}]\label{acyclic}
Every acyclic digraph has an acyclic ordering of its vertices.
	\end{proposition}

\begin{proof}[Proof of Lemma~\ref{example-lem}]
Given a $k$-graph $F$ with $f$ vertices, let $\tau=(v_1, \dots, v_{f})$ be a vanishing ordering of $V(F)$. 
	We construct a $k$-edge-coloring $(k-1)$-digraph $D$ on $V(F)$  as follows. 
	For any $e\in E(F)$, if $e=\llbracket v_{i(1)},\dots, v_{i(k)} \rrbracket$ under  $\tau$, the directed tight $(k-1)$-uniform cycle $\vec C_k^{(k-1)}=(v_{i(1)},\dots, v_{i(k)})$ is present in $D$. Furthermore, for every $j \in [k]$, the directed edge $(v_{i(j)}, v_{i(j+1)}, \dots, v_{i(j+k-2)}) $ is colored $(j-1)$ for every $j\in [k]$ (where the subscripts are taken modulo $k$). Since $\tau$ is a vanishing ordering of $V(F)$, we obtain that $D$ is a simple $(k-1)$-digraph and each $k$-edge of $F$ corresponds to a $\vec C_k^{(k-1)}$ with edges colored $0, 1,\dots  k-1$ in this ordering. Let $D_{0,1}$ denote the subdigraph of $D$ containing all directed edges colored with $0$ or $1$. Observe that for each edge $e=\llbracket v_{i(1)},\dots, v_{i(k)} \rrbracket\in E(F)$, we have $(v_{i(1)}, v_{i(2)}, \dots, v_{i(k-1)}), (v_{i(2)}, v_{i(3)}, \dots, v_{i(k)})\in A(D_{0,1})$, which implies that every
	directed edge in the transitive digraph $T(D_{0,1})$ is directed from a vertex with a small index to a vertex  with a large index under the ordering $\tau$.  Hence, $T(D_{0,1})$ has no directed cycles.

	Next, given an ordering of $V(F)$, suppose that there exists a  $k$-edge-coloring simple $(k-1)$-digraph $D$ with colors $0,1,\dots, k-1$ satisfying the properties given in the lemma. By symmetry, we may assume that the transitive digraph $T(D_{0,1})$ is acyclic (otherwise,
	we cyclically rotate the colors to satisfy this). By Proposition~\ref{acyclic},  $T(D_{0,1})$ has an acyclic ordering $\tau'$ of $V(T(D_{0,1}))$.
By the definition of $T(D_{0,1})$, $\tau'$ is a  vanishing ordering of $V(F)$.
\end{proof}

Now we give a proof of Theorem~\ref{example-thm} using Lemma~\ref{example-lem} and Theorem~\ref{1/kthm}.
%
%

\begin{proof}[Proof of Theorem~\ref{example-thm}]
Let $F^{k}_t$ be the $k$-graph given in Theorem~\ref{example-thm}. We first apply Lemma~\ref{example-lem} to show that $F^{k}_t$ has no
	vanishing ordering of $V(F^{k}_t)$.
	Consider a $(k-1)$-digraph $D$ as described in the statement of Lemma~\ref{example-lem}.
Due to symmetry, it is allowed to assume that $(a_1, a_2, \dots a_{k-1})\in A(D)$ and is colored with $0$. Set $x\in \{b,c,d\}$.  Since each $k$-edge of $F^{k}_t$ corresponds to a directed tight $k$-uniform cycle in $D$ with edges colored $0, 1,\dots,  k-1$, by cyclic symmetry, we obtain that $(a_i, \dots, a_{k-1}, x_0, \dots, x_{i-2})\in A(D)$ with color $(i-1)$ for each $i\in [k-1]$,  and 
 $(x_0, x_1, \dots,  x_{k-2})\in A(D)$   with color $k-1$. Moreover, we also obtain that $(x_{\ell}, x_{\ell+1}, \dots,  x_{\ell+k-2})\in A(D)$ with color $\ell-1 \pmod k$ for each $0\le \ell \le t-k+2$, $(b_{t-k+3},  \dots, b_t, c_t), (c_{t-k+3},  \dots, c_t, d_t)$, and $(d_{t-k+3},  \dots, d_t, b_t)\in A(D)$ with color $t-k+2 \pmod k$. 
	
Given $\beta\in \mathbb Z_k$, let  $D_{\beta}$ denote the sub-digraph of $D$ containing all directed $(k-1)$-edges colored with $\beta$. For each $j\in \{t-k+3, t-k+4, \dots, t\}$, if $\beta \equiv j \pmod k$, then the transitive digraph $T(D_{\beta})$ always contains a directed cycle formed by $(b_t, c_t), (c_t, d_t), (d_t, b_t)$. For simplicity, let $D'=D_{t-k+2 \pmod k}\cup D_{t+1 \pmod k}$. Observe that the following directed edges
\[(b_{t-k+2}, \dots,b_{t-1}, b_t), (c_{t-k+2}, \dots, c_{t-1}, c_t), \text{~and ~} (d_{t-k+2}, \dots, d_{t-1}, d_t) \] 
\[
(c_t, b_{t-k+2}, \dots, b_{t-1}), (d_t, c_{t-k+2}, \dots, c_{t-1}), \text{~and ~} (b_t, d_{t-k+2}, \dots, d_{t-1})
\]
 all belong to $D'$. Therefore, the transitive digraph $T(D')$  also contains a directed cycle formed by
 \[
(b_{t-1}, b_t), (b_t, d_{t-1}), (d_{t-1}, d_t), (d_t, c_{t-1}), (c_{t-1}, c_t), (c_t, b_{t-1}).
 \]
Hence, for each pair $\{\beta, \beta+1\}\subset \mathbb Z_{k}$, the transitive digraph $T(D_{\beta,\beta+1})$ always contains a  directed cycle.
	By Lemma~\ref{example-lem},  $F^{(k)}_t$ has no
	vanishing ordering of $V(F)$.
	
	Next, we claim that $F^{(k)}_t$ also satisfies  the property ($\spadesuit$) of Theorem~\ref{1/kthm}. 
	For simplicity, let $e_1=\{d_{t-k+2},\dots, d_t,b_t\}$, $e_2=\{d_{t-k+1}, \dots d_{t-1}, d_t\}$ and $S=e_1\cap e_2=\{d_{t-k+2}, \dots, d_t\}$.
Furthermore, let $F_1$ be the spanning subhypergraph of $F^{(k)}_t$ with the only edge $e_1$,
and  $F_2$  be the spanning subhypergraph of $F^{(k)}_t$ by removing the edge $e_1$.
Clearly, $\{S\} = \partial F_1 \cap \partial F_2$.
Next for each $\{i,j\}\in [k]^2$ with $i<j$, we would find an ordering $\tau_{i,j}$ of $V(F^{(k)}_t)$ such that $\tau_{i,j}$ is vanishing both for $F_1$ and $F_2$, and  the $(k-1)$-set $S$ is $i$-type w.r.t $F_1$ and $j$-type w.r.t  $F_2$ under the ordering $\tau_{i,j}$.
	We first consider a partition $\{X_1, X_2, \dots, X_k\}$ of $V(F^{(k)}_t)$ with 
	\[
	X_{\ell}=\{x_r\in V(F^{(k)}_t): r\equiv t+\ell \pmod k,~~ x\in \{a,b,c,d\}, ~~0\le r\le  t \} \text{~for~} \ell\in [k].
	\]
	Set $Y_\ell=X_{\ell}\setminus \{b_{t-k+\ell}, d_{t-k+\ell}\}$ for $\ell\in [k]$.
	
	When $i>1$ and $j-i\ge 2$, we consider an ordering  $\tau_{i,j}$ of $V(F^{(k)}_t)$ that contains, in turn, all vertices of $X_2$, then all vertices of $X_3$, and so on, up to all vertices of $X_i$, then followed by the ordering $(b_{t-k+i+1}, \dots, b_{t-k+j-1}$, $b_t, d_{t-k+i+1},$ $ \dots, d_{t-k+j-1})$, then all vertices of $Y_{i+1}$, and so on, up to all vertices of $Y_{j-1}$,  then all vertices of $Y_{k}$,  then the vertex $d_t$, then all vertices of $X_1$, then all vertices of $X_j$, then all vertices of $X_{j+1}$, and so on, up to all vertices of $X_{k-1}$. The ordering of elements inside the sets $X_\ell$ and $Y_\ell$ is arbitrary. Note that $\tau_{i,j}$ is vanishing for $F_1$ since $E(F_1)=\{e_1\}$. Under the ordering $\tau_{i,j}$, we have
\[
\begin{split}
& e_1=(d_{t-k+2}, \dots, d_{t-k+i}, b_t~~~~~~, d_{t-k+i+1}, \dots, d_{t-k+j-1}, d_t, d_{t-k+j}, \dots, d_{t-1}), \text{~and~}\\
& e_2=(d_{t-k+2}, \dots, d_{t-k+i}, d_{t-k+1}, d_{t-k+i+1}, \dots, d_{t-k+j-1}, d_t, d_{t-k+j}, \dots, d_{t-1}) 
\end{split}
\]
Therefore, the set $S$ is $i$-type w.r.t. $e_1$ and  $j$-type w.r.t. $e_2$.  Set $e_3=\{b_{t-k+2},\dots, b_t, c_t\}$ and $e_4=\{c_{t-k+2},\dots, c_t, d_t\}$. Observe that each $k$-edge in  $F_2\setminus \{e_3, e_4\}$, contains exactly one vertex from $X_{\ell}$ for $\ell \in [k]$. In particular, given $e\in E(F_2)\setminus \{e_3, e_4\}$ and $0\le r\le t-k$, $e\setminus \{x_{r+\ell: r\equiv t \pmod k}\}$ is $(\ell-1)$-type w.r.t. $F_2$ for $2\le \ell\le j-1$, $e\setminus \{x_{r+k: r\equiv t \pmod k}\}$ is $(j-1)$-type w.r.t. $F_2$, $e\setminus \{x_{r+1: r\equiv t \pmod k}\}$ is $j$-type w.r.t. $F_2$, and $e\setminus \{x_{r+\ell: r\equiv t \pmod k}\}$ is $(\ell+1)$-type w.r.t. $F_2$ for $j\le \ell\le k-1$.
Furthermore, the order of $e_3$ is $(b_{t-k+2}, \dots, b_{t-k+j-1},b_t, c_t, b_{t-k+j}, \dots, b_{t-1})$, and 
	the order of $e_4$ is $(c_{t-k+2}, \dots, c_{t-k+j-1}, c_t, d_t$, $c_{t-k+j}, \dots, c_{t-1})$. Therefore, $\tau_{i,j}$ is also a vanishing ordering of  $V(F_2)$ and the set $S$ is $j$-type  w.r.t. $F_2$ under the ordering $\tau_{i,j}$.
	
	When $i=1$ and $j-i\ge 2$, we consider any ordering  $\tau'_{i,j}$ of $V(F^{(k)}_t)$ that contains, in turn, the ordering $(b_{t-k+2}, \dots,$ $ b_{t-k+j-1}, b_t, d_{t-k+2}, \dots, d_{t-k+j-1})$,  then all vertices of $Y_2$, $\dots$, then all vertices of $Y_{j-1}$,  then all vertices of $Y_{k}$,  then vertex $d_t$, then all vertices of $X_1$, then all vertices of $X_j$, $\dots$, and then all vertices of $X_{k-1}$.
	If $i=1$ and $j=2$, then we consider $\tau'_{1,2}$ in turn contains the vertex $b_t$,  then all vertices of $Y_{k}$,  then vertex $d_t$, then all vertices of $X_1$, $\dots$, and then all vertices of $X_{k-1}$.
	Similarly, we can easily verify that the ordering $\tau'_{i,j}$ is vanishing both for $F_1$ and $F_2$
	and the set  $S$ is $i$-type w.r.t. $F_1$ and $j$-type w.r.t. $F_2$.  
	
	When $i>1$ and $j=i+1$, we consider any ordering  $\tau''_{i,j}$ of $V(F^{(k)}_t)$ that contains, in turn, all vertices of $X_2$, then all vertices of $X_3$, $\dots$, then all vertices of $X_i$, then the vertex $b_t$, then all vertices of $Y_{k}$,  then the vertex $d_t$, then all vertices of $X_1$, then all vertices of $X_{i+1}$, $\dots$, and then all vertices of $X_{k-1}$. Similarly, we can easily verify that the ordering $\tau''_{i,j}$ is vanishing both for $F_1$ and $F_2$
	and the set  $S$ is $i$-type w.r.t. $F_1$ and $j$-type w.r.t. $F_2$. Hence, $F^{(k)}_t$ also satisfies  the property ($\spadesuit$) of Theorem~\ref{1/kthm}.
\end{proof}

\section{Concluding Remarks}\label{remarks}
Theorem~\ref{1/kthm} provides a sufficient condition for $k$-graphs $F$ satisfying $\pi_{k-2}(F)=k^{-k}$.
Although we do not currently have a complete characterization for all $k$-graphs $F$ with $\pi_{k-2}(F)=k^{-k}$, this sufficient condition is likely to be  close to the complete characterization due to the following result.

\begin{theorem}\label{1/kthm-conj}
	Given $k\ge 3$, let $F$ be a $k$-graph that does not satisfy the following condition:
	\begin{enumerate}[label=$(\roman*)$]
	\item[{\rm ($\spadesuit^*$)}] For each pair $\{i,j\}\in [k]^2$ with $i<j$, $F$ can always be partitioned into two spanning subhypergraphs $F^1_{i,j}$ and $F^2_{i,j}$ such that there exists an ordering of $V(F)$ that is vanishing both for $F^1_{i,j}$ and $F^2_{i,j}$,
and for each pair $e_1\in E(F^1_{i,j})$, $e_2\in E(F^2_{i,j})$ with $|e_1\cap e_2|=k-1$,  $e_1\cap e_2$ is either 
same $\ell$-type w.r.t. $F^1_{i,j}$ and $F^2_{i,j}$ for some $\ell\in [k]$, or $e_1\cap e_2$ is $i$-type w.r.t. $F^1_{i,j}$ and $e_1\cap e_2$ is $j$-type w.r.t. $F^2_{i,j}$.
	\end{enumerate}
	Then $\pi_{k-2}(F)\ge 3(k+1)^{-k}>k^{-k}$.
\end{theorem}

\begin{proof}
	Suppose that  there exists a pair $\{i', j'\}\in [k]^2$ with $i'<j'$ such that $k$-graph $F$ does not satisfy the property ($\spadesuit^*$). We will prove $\pi_{k-2}(F)\ge 3(k+1)^{-k}$ using Theorem~\ref{low-k-graphs}.
	Let $\mathscr P =\{(1,2,\dots,k), (1, \dots,i'-1, k+1, i'+1, \dots, k), (1, \dots,j'-1, k+1, j'+1, \dots, k)\}$ be a subset of $[k+1]^{[k]}$.
	For every $n\in \mathbb N$ and $\psi: [n]^{k-1} \to [k+1]$, we consider the 
 $k$-graph $H$ with vertex set $[n]$ and edge set 
 	\[
 E(H)=\{e=\llbracket {i_1,i_2, \dots, i_\ell}\rrbracket \in [n]^k: (\psi(e\setminus \{i_{1}\}), \psi(e\setminus \{i_{2}\}), \dots, \psi(e\setminus \{i_{k}\}))\in \mathscr P\}.
 \] 
Observe that each subhypergraph of $H$ satisfies the property ($\spadesuit^*$) for $\{i', j'\}$.
Therefore, $H$ is $F$-free. By Theorem~\ref{low-k-graphs}, we have $\pi_{k-2}(F)\ge|\mathscr P|/(k+1)^k= 3(k+1)^{-k}$.
\end{proof}

Inspired by  Theorem~\ref{1/kthm-conj}, we have the following conjecture.
\begin{conjecture}\label{conj1}
Given $k\ge 3$, let $F$ be a $k$-graph satisfying the following conditions:
\begin{enumerate}[label=$(\roman*)$]
	\item[{\rm ($\clubsuit$)}] $F$ has no vanishing ordering of $V(F)$; 
	\item[{\rm ($\spadesuit^*$)}] For each pair $\{i,j\}\in [k]^2$ with $i<j$, $F$ can always be partitioned into two spanning subhypergraphs $F^1_{i,j}$ and $F^2_{i,j}$ such that there exists an ordering of $V(F)$ that is vanishing both for $F^1_{i,j}$ and $F^2_{i,j}$,
	and for any two edges $e_1\in E(F^1_{i,j})$, $e_2\in E(F^2_{i,j})$ with $|e_1\cap e_2|=k-1$,  $e_1\cap e_2$ is either 
	same $\ell$-type w.r.t. $F^1_{i,j}$ and $F^2_{i,j}$ for some $\ell\in [k]$, or $e_1\cap e_2$ is $i$-type w.r.t. $F^1_{i,j}$ and $e_1\cap e_2$ is $j$-type w.r.t. $F^2_{i,j}$.
\end{enumerate}
Then $\pi_{k-2}(F)=k^{-k}$.
\end{conjecture}

Garbe, Kr\'al' and Lamaison~\cite{1/27} also asked the following problem for  $k=3$.
\begin{problem}
Does there exist $\eps>0$ such that   $\pi_{k-2}(\cdot)$ jumps from  $k^{-k}$ to at least $k^{-k} +\eps$?
\end{problem}

If the Conjecture~\ref{conj1} is true, then we could  easily obtain that $\pi_{k-2}(\cdot)$ will jump from  $k^{-k}$ to at least $3(k+1)^{-k}$.
Moreover, compared to  Problem~\ref{problem2}, it is natural to ask
the following problem.

\begin{problem}
Is there a $k$-graph $F$ such that $\pi_{k-2}(F)$ is equal to or arbitrarily close to $3(k+1)^{-k}$?
\end{problem}

\bibliographystyle{abbrv}
\bibliography{ref}
\end{document}